\documentclass[11pt,a4paper]{amsart}
\usepackage{amsmath, graphicx, subfigure, epsfig,amssymb,cite}
\usepackage{xcolor,mathtools}
\usepackage[normalem]{ulem}
\numberwithin{equation}{section}

\usepackage{hyperref}
\hypersetup{
    colorlinks=true,
    linkcolor=blue,
    filecolor=magenta,      
    urlcolor=cyan,
}

\newcommand{\CG}{\mathcal G}
\newcommand{\BP}{\mathbb P}
\newcommand{\chx}{\check x}  
\newcommand{\chy}{\check y}  
\newcommand{\dtp}{\omega}

\newcommand{\e}{\epsilon}
\newcommand{\ga}{\gamma}
\newcommand{\de}{\delta}
\newcommand{\br}{\mathbb{R}}
\newcommand{\BN}{\mathbb{N}}
\newcommand{\ik}{\varphi}
\newcommand{\pa}{\partial}

\newcommand{\al}{\alpha}
\newcommand{\la}{\lambda}
\newcommand{\om}{\omega}

\newcommand{\be}{\begin{equation}}
\newcommand{\ee}{\end{equation}}
\newcommand{\bs}{\begin{split}}
\newcommand{\es}{\end{split}}
\newcommand{\dd}{\text{d}}
\newcommand{\nrec}{N_\e^{\text{rec}}}
\newcommand{\frec}{f_\e^{\text{rec}}}

\newcommand{\CA}{\mathcal{A}}
\newcommand{\CB}{\mathcal{B}}

\newcommand{\us}{\mathcal X}
\newcommand{\vs}{\mathcal V}
\newcommand{\s}{\mathcal S}
\newcommand{\R}{\mathcal R}

\newcommand{\BZ}{\mathbb Z}
\newcommand{\Eb}{\mathbb E}

\newcommand{\CY}{\mathcal Y}
\newcommand{\CZ}{\mathcal Z}

\newcommand{\DTB}{\text{DTB}}

\newtheorem{theorem}{Theorem}[section]
\newtheorem{lemma}[theorem]{Lemma}
\newtheorem{corollary}[theorem]{Corollary}

\theoremstyle{definition}
\newtheorem{assumptions}[theorem]{Assumption}
\newtheorem{definition}[theorem]{Definition}

\begin{document}

\title[Analysis of noise]{Analysis of reconstruction from noisy discrete generalized Radon data}
\author[A Katsevich]{Alexander Katsevich$^1$}
\thanks{$^1$This work was supported in part by NSF grant DMS-1906361. Department of Mathematics, University of Central Florida, Orlando, FL 32816 (Alexander.Katsevich@ucf.edu). }

\begin{abstract} 
We consider a wide class of generalized Radon transforms $\mathcal R$, which act in $\mathbb{R}^n$ for any $n\ge 2$  and integrate over submanifolds of any codimension $N$, $1\le N\le n-1$. Also, we allow for a fairly general reconstruction operator $\mathcal A$. The main requirement is that $\mathcal A$ be a Fourier integral operator with a phase function, which is linear in the phase variable. We consider the task of image reconstruction from discrete data $g_{j,k} = (\mathcal R f)_{j,k} + \eta_{j,k}$. We show that the reconstruction error $N_\epsilon^{\text{rec}}=\mathcal A \eta_{j,k}$ satisfies 
$N^{\text{rec}}(\chx;x_0)=\lim_{\epsilon\to0}N_\e^{\text{rec}}(x_0+\epsilon\check x)$, $\check x\in D$.
Here $x_0$ is a fixed point, $D\subset\mathbb{R}^n$ is a bounded domain, and $\eta_{j,k}$ are independent, but not necessarily identically distributed, random variables. $N^{\text{rec}}$ and $N_\e^{\text{rec}}$ are viewed as continuous random functions of the argument $\check x$ (random fields), and the limit is understood in the sense of probability distributions. Under some conditions on the first three moments of $\eta_{j,k}$ (and some other not very restrictive conditions on $x_0$ and $\mathcal A$), we prove that $N^{\text{rec}}$ is a zero mean Gaussian random field and explicitly compute its covariance. We also present a numerical experiment with a cone beam transform in $\mathbb{R}^3$, which shows an excellent match between theoretical predictions and simulated reconstructions.
\end{abstract}

\maketitle

\section{Introduction}

\subsection{Prior work. Main results} Tomographic imaging, which includes X-ray, ultrasound, seismic, and many other tomographies, is commonly used in a wide range of applications such as medical, industrial, security, and others. The task of tomographic image reconstruction can be formulated as follows. Let $\R:\us\to\vs$ denote a linear integral transform, where $\us\subset\br^n$ is the image domain and $\vs\subset\br^n$ is the data domain. Let $f$ be a function supported in $\us$, which represents the object being scanned. Tomographic data represent integrals of $f$ along a family of manifolds (e.g., lines, spheres, etc.). The manifolds are parametrized by points $\dtp\in \vs$, so the data are $(\R f)(\dtp)$, $\dtp\in\vs$. In a continuous setting, the goal is to reconstruct $f(x)$, $x\in\us$ (or some information about $f$, e.g., its singularities) from $(\R f)(\dtp)$, $\dtp\in\vs$. 

In practice, data are always discrete and contain noise. So the task becomes to reconstruct an approximation to $f$ from discrete data $g_j:=(\R f)(\e j)+\eta_j$, $\e j\in\vs$, $j\in\BZ^n$. Here $\e>0$ represents the data step size and $\eta_j$’s represent additive noise. For simplicity, we assume that the step size is the same along each direction. Usually, two most important factors that affect quality of reconstructed images are data discretization and strength of noise in the data. 

Effects of data discretization on various aspects of image quality, e. g. spatial resolution, aliasing artifacts, are now well-understood \cite{nat93, pal95, far04, stef20}. In particular, in a series of articles, the author developed an approach, called \textit{local reconstruction analysis} (LRA), to analyze the resolution with which the singularities (i.e., jump discontinuities) of $f$ are reconstructed from discrete tomographic data in the absence of noise, see \cite{Katsevich2021a, Katsevich2023a, Katsevich_2024_BV} and references therein. In those articles the singularities of $f$ are assumed to lie on a smooth curve in $\br^2$ (and surface in $\br^n$) or a rough curve in $\br^2$, denoted $\s$.  In \cite{Katsevich_aliasing_2023}, LRA theory was advanced to include analysis of aliasing (ripple artifacts) in $\br^2$ at points away from $\s$.

The main idea of LRA is to obtain a simple formula to accurately approximate an image reconstructed from discrete data, $\frec$, in an $\e$-neighborhood of a point, $x_0$. For example, let $f$ be a real-valued function in $\br^2$, $\s\subset\br^2$ be a smooth curve, and $f$ have a jump across $\s$. Let $\frec$ be a filtered backprojection (FBP) reconstruction of $f$ from discretely sampled Radon transform (RT) data with sampling step size, $\e$. Under some mild conditions on $\s$, it is shown in \cite{Katsevich2022a,Katsevich2023a} that one has
\be\label{DTB new use}
\lim_{\e\to0}\frec(x_0+\e\check x)=\Delta f(x_0)\DTB(\check x;x_0),\ x_0\in\s,
\ee
where the limit is uniform with respect to $\check x\in D$ for any bounded set $D$. Here $\Delta f(x_0)$ is the value of the jump of $f$ across $\s$ at $x_0$ and DTB, which stands for the \textit{Discrete Transition Behavior}, is an easily computable function independent of $f$. The DTB function depends only on the curvature of $\s$ at $x_0$. 
When $\e$ is sufficiently small, the right-hand side of \eqref{DTB new use} is an accurate approximation of $\frec$ and the DTB function describes accurately the smoothing of the singularities of $f$ in $\frec$.

As is seen from \eqref{DTB new use}, {\it LRA provides a uniform approximation to $\frec(x)$ in domains of size $\sim\e$, which is of the same order of magnitude as the sampling step size $\e$.} Given the DTB function, one can study local properties of reconstruction from discrete data, such as spatial resolution and strength of artifacts. 

The first extension of LRA to reconstruction from noisy data has been done in \cite{AKW2024}. The authors consider the model $g = \R f + \eta$, where $\R$ is the classical 2D RT, $\R f$ is the measured sinogram with entries $(\R f)_{j,k}$, and $\eta$ is the noise. The entries, $\eta_{j,k}$, of $\eta$ are assumed to be independent, but not necessarily identically distributed. Due to linearity of $\R^{-1}$, the reconstruction error is $\R^{-1}\eta$. Let $N_\e^{\text{rec}}(x)$ denote the reconstruction error, i.e. the FBP reconstruction only from $\eta_{j,k}$. Similarly to previous works on LRA (cf. also \eqref{DTB new use}), the authors consider $O(\e)$-sized neighborhoods around any $x_0\in \br^2$. Let $C(D)$ be the space of continuous functions on $D$, where $D\subset\br^2$ is a bounded domain. The main result of \cite{AKW2024} is that, under suitable assumptions on the first three moments of the $\eta_{j,k}$, the boundary of $D$, and $x_0$, the following limit exists:
\be\label{noise rec 0}
N^{\text{rec}}(\chx;x_0)=\lim_{\e\to0}N_\e^{\text{rec}}(x_0+\e\chx),\ \check x\in D.
\ee
Here, $N^{\text{rec}}$ and $N_\e^{\text{rec}}$ are viewed as $C(D)$-valued random variables (random fields), and the limit is understood in the sense of probability distributions. It is proven also that $N^{\text{rec}}(\chx;x_0)$ is a zero mean Gaussian random field (GRF) and its covariance is computed explicitly. The paper also contains numerical experiments, which show an excellent match between theoretical predictions and simulated reconstructions.

In this paper we extend the results of \cite{AKW2024} to a wide class of generalized RTs $\R$, which act in $\br^n$ for any $n\ge 2$  and integrate over submanifolds of any codimension $N$, $1\le N\le n-1$. Also, we allow for a fairly general reconstruction operator $\CA$. The main requirement is that $\CA$ be a Fourier integral operator (FIO) with a phase function, which is linear in the phase variable. Similarly to \cite{AKW2024} we consider the model $g = \R f + \eta$ and show that the reconstruction error $N_\e^{\text{rec}}=\CA \eta_{j,k}$ satisfies \eqref{noise rec 0}, where $D\subset\br^n$ is a bounded domain and $\eta_{j,k}$ are independent, but not necessarily identically distributed, random variables. As in \cite{AKW2024}, $N^{\text{rec}}$ and $N_\e^{\text{rec}}$ are viewed as $C(D)$-valued random variables, and the limit is understood in the sense of probability distributions. We prove that $N^{\text{rec}}$ is a zero mean GRF and explicitly compute its covariance. We also present a numerical experiment with a cone beam transform in $\br^3$, which shows an excellent match between theoretical predictions and simulated reconstructions.

\subsection{Significance} 
Equations \eqref{DTB new use} and \eqref{noise rec 0} together provide an explicit and accurate local approximation to the image reconstructed from noisy discrete data, $\frec=\CA(\R f+\eta)$. Eq. \eqref{DTB new use} describes the effect of edge smoothing due to data discretization, and  \eqref{noise rec 0} describes the reconstruction error due to random noise. The combined approximation is uniform over domains of size $\sim\e$, i.e. precisely at the scale of {\it native} resolution enabled by available data. Results of this kind are more useful for localized analysis of tomographic reconstruction than more common global analyses, which estimate some global error norm (e.g., $L^2$). Formulas provided by LRA, such as \eqref{DTB new use} and \eqref{noise rec 0}, can be used to study resolution of reconstruction, perform statistical inference about detectability of small details in a reconstructed image, and for many other tasks.

As a first example we consider the task of detection and assessment of lung tumors in CT images (see also \cite{AKW2024}). Typically, malignant lung tumors have rougher boundaries than benign nodules \cite{dhara2016combination,dhara2016differential}. Therefore the roughness of the nodule boundary is a critical factor for accurate diagnosis. Reconstructions from discrete X-ray CT data produce images in which the singularities are smoothed to various degrees. A rough boundary of the tumor may appear smoother in the reconstructed image than it really is. This can lead to a cancerous tumor being misdiagnosed as a benign nodule. Likewise, due to the random noise in the data, the smooth boundary of a benign nodule may appear rougher than it actually is, which may again result in misdiagnosis. This example illustrates the need to accurately quantify the effects of both data discretization and random noise on {\it local} (i.e., near the tumor boundary) tomographic reconstruction.

Another example is the use of CT by the energy industry for imaging of rock samples extracted from wells. The reconstructed image is segmented to identify as accurately as possible the pore space (i.e., voids) inside the sample. This is very challenging, because pore boundaries are fractal, they possess features that are below the scanner native resolution. Then one uses numerical fluid flow simulations {\it inside the identified pore space} to compute permeability of the sample \cite{bss18}. The obtained results are used for formation evaluation and improving of oil recovery \cite{sha17}. Thus, the key step that affects the accuracy of the entire workflow from scanning to fluid flow simulation is CT image segmentation. Errors in locating pore boundaries and, consequently, inaccurate pore space identification lead to incorrect flow simulation and errors in computed permeability \cite{Chhatre2017, Saxena2019, Saxena2019a}. Therefore, as in the previous example, precise quantification of the local (near the boundaries) effects of data discretization and noise on the reconstructed image is of paramount importance.  

\subsection{Related literature. Organization of the paper}
We now discuss related existing literature concerning reconstruction in a stochastic setting and compare it with our findings. A kernel-type estimator of $f$ has been derived in \cite{Korostelev1991,Tsybakov_92} in the setting of a tomography problem with additive random errors in observations. Its minimax optimal rate of convergence to $f$, i.e. the ground truth function, has been established both at a fixed point and in a global $L^2$ norm. The rate of convergence for the maximal deviation of an estimator from its mean has been obtained for similar kernel-type estimators in \cite{Bissantz2014}.  The accuracy of pointwise asymptotically efficient kernel estimator in terms of minimax risk of a probability density $f$ from noise-free RT data sampled on a random grid has been derived in \cite{Cavalier_98,Cavalier_00}. 

An essential common feature of all the works cited above is that convergence is established by incorporating into the reconstruction/estimation algorithm of additional smoothing at a scale $\delta$ significantly larger than the data step size. This smoothing leads to a notable loss of resolution in practical applications. This is also the reason why the cited works assume that the function $f$ being estimated is sufficiently smooth. To compare with our results, we mention two points. First, none of these papers obtain the probability distribution of the reconstructed noise (i.e., the error in the reconstruction), either pointwise or in a domain. Second, our reconstruction is performed and investigated at {\it native} resolution.

Reconstruction in the framework of Bayesian inversion has been explored as well \cite{Monard_19,Siltanen2003,Lassas_09}. Global inversion in the case when continuous data are corrupted by a Gaussian white noise is investigated in \cite{Monard_19}. Using a Gaussian prior (i.e., with Tikhonov regularization), the authors establish the asymptotic normality of the posterior distribution and of the MAP estimator for observables of the kind $\int f(x)\psi(x)\dd x$, where $\psi\in C^\infty$ is a test function. This means that reconstruction is investigated at the scales $\sim1$. 
In contrast, our results quantify the pointwise reconstruction error uniformly over regions of size $O(\e)$. Various other aspects of Bayesian inversion are discussed in \cite{Siltanen2003,Lassas_09}.

Analysis of reconstruction errors using semiclassical analysis is developed in \cite{stef_23}. The goal of the paper is to analyze empirical spatial mean and variance of the noise in the inversion for a single experiment in the limit as the sampling rate $\e$ goes to zero. In contrast, the emphasis of our paper is on the study of the reconstruction error across multiple reconstructions. We obtain the entire probability density function (PDF) of the error in the limit as $\e\to0$ (as opposed to only the mean and variance).

Analysis of noise in reconstructed images is an active area in more applied research as well (see \cite{Noo_noise_2008, Divel2020} and references therein). In this research, the proposed methodologies mostly  combine numerical and semi-empirical approaches. Theoretical derivation of the reconstructed noise PDF in small neighborhoods is not provided.

The paper is organized as follows. In section~\ref{sec:a-prelims}, we describe the setting of the problem and main assumptions. In section~\ref{sec:main_res}, we state the main results, which are broken down into three theorems. In section~\ref{sec:examples} we illustrate how to check the key assumptions for a few common CT geometries. The beginning of the proof of the first result, theorem~\ref{thm:Lyapunov1d}, is in section~\ref{sec:prf beg}. In this section we study the contribution of the leading order term of $\CA$ to the reconstructed image, $N_\e^{\text{rec}}(x_0+\e\chx)$. In section~\ref{sec:LOTs} we study the contribution of the lower order terms of $\CA$ to the reconstruction. In section~\ref{sec:Lyapunov} we finish the proof of theorem~\ref{thm:Lyapunov1d} and prove theorem~\ref{thm:Lyapunov_nd}, our second main result. In section~\ref{sec:thrd thrm prf} we prove our last result, theorem~\ref{GRF_thm}. A numerical experiment is described in section~\ref{sec:numerics}. Finally, the proofs of some auxiliary lemmas are in Appendices \ref{sec:two aux res}--\ref{sec:prf outline}.

\section{Setting of the problem and main assumptions}\label{sec:a-prelims}

A general Radon-type integral transform $\R:\us\to\vs$, where $\us,\vs\subset\br^n$ are some domains, can be written as a Fourier Integral Operator (FIO) with a phase function, which is linear in the frequency variable. Except for a small number of special cases, exact inversion formulas are usually not available. Hence one may be interested in applying an approximate inversion formula in the form of a parametrix. In other cases, one may be interested in enhanced-edge reconstruction \cite{rk, fbh01, hl12, louis16}. Therefore we consider a general reconstruction operator $\CA:\vs\to\us$, which is also an FIO with a linear phase function
\be\label{FIO def}
(\CA g)(x)=(2\pi)^{-N}\int_{\br^N}\int_{\CY}\int_{\CZ} a(x,(y,z),\xi) e^{i\xi\cdot\Phi(x,(y,z))}g(y,z)\dd z\dd y\dd\xi.
\ee
where $\Phi:\us\times\vs\to\br^N$ is a smooth function. 
\begin{assumptions}[Properties of the domains $\us$, $\vs$]\label{ass:domains}
$\hspace{1cm}$ 
\begin{enumerate}
\item $\us\subset\br^n$ and $\vs\subset\br^n$ are bounded domains. 
\item $\vs=\CY\times\CZ$ for some domains $\CY\subset\br^{n-N}$, $\CZ\subset\br^N$.
\end{enumerate}
\end{assumptions}

In our setting, $x\in\us$ are points in the image domain, and $(y,z)\in\vs$ are points in the data (projection) domain. The pair $(y,z)\in\vs$ was denoted $\dtp$ in the Introduction. The representation of $\dtp$, $\dtp=(y,z)$, as a pair is analogous to the classical RT convention $(\al,p)\in S^{n-1}\times\br$. In our case, $y$ is the analog of the variable $\al$ and $z$ is a (multidimensional) analog of the variable $p$.

Denote $\BN_0:=\BN\cup \{0\}$. For convenience, throughout the paper we use the following convention. If a constant $c$ is used in an equation or inequality, the qualifier ‘for some $c>0$’ is assumed. If several $c$ are used in a string of (in)equalities, then ‘for some’ applies to each of them, and the values of different $c$’s may all be different.

\begin{definition} Pick any $\ga\in\br$. The set $S^{\ga}(\us\times\vs\times(\br^N\setminus 0))$ is the vector space of all functions $a(x,\dtp,\xi):\us\times\vs\times\br^N\to\br$ such that
\be\bs\label{symb def}
&|\pa_\xi^m \pa_{(x,\dtp)}^k a(x,\dtp,\xi)|\le c_{m,k}(1+|\xi|)^{\ga-|m|},\\ 
&\hspace{0cm} (x,\dtp)\in\us\times\vs,|\xi|\ge 1,m\in\BN_0^N,k\in\BN_0^{2n};\\ 
&\pa_{(x,\dtp)}^k a\in L^1_{loc}(\us\times\vs\times \br^N),\ k\in\BN_0^{2n}.
\end{split}
\ee
\end{definition}

\begin{assumptions}[Properties of $a$, the amplitude of $\CA$]\label{ass:grt} 
$\hspace{1cm}$ 
\begin{enumerate}
\item One has
\be\bs
&a\in S^{\ga}(\us\times\vs\times(\br^N\setminus 0)) \text{ for some }\ga>-N/2.
\end{split}
\ee
\item The principal symbol of $\CA$, denoted $a_0$, is homogeneous and satisfies: 
\be\bs\label{a0_props}
&a_0(x,\dtp,r\xi)=r^\ga a_0(x,\dtp,\xi),\
r>0,\,(x,\dtp)\in\us\times\vs,\,\xi\in\br^N\setminus 0,\\
&a_0\in S^{\ga}(\us\times\vs\times(\br^N\setminus 0)) .
\end{split}
\ee
\item
$a-a_0\in S^{\ga^\prime}(\us\times\vs\times(\br^N\setminus0))$ for some $\ga^\prime<\ga$.
\end{enumerate}
\end{assumptions}

\begin{assumptions}[Properties of $\Phi$]\label{ass:Phi} 
$\hspace{1cm}$ 
\begin{enumerate}
\item $\Phi:U\to\br^N$ is a smooth function, where $U\subset\br^{2n}$ is a neighborhood of $\us\times\vs$.
\item For each $x\in\us$ and $y\in\CY$ there exists a unique point $z=\Psi(x,y)\in\CZ$ such that $\Phi(x,(y,\Psi(x,y)))\equiv0$.
\item $|\text{det}(\pa_z\Phi(x,(y,z)))|\ge c>0$ for any $z=\Psi(x,y)$ and $(x,y)\in\us\times\CY$. 
\end{enumerate}
\end{assumptions}

Denote $u:=(x,y)$, $a_0(u,z,\xi):=a_0(x,(y,z),\xi)$, and similarly for all other functions which depend on $x$ and $y$. Assumption~\ref{ass:Phi} and the implicit function theorem \cite[Theorem 3.3.1]{IFT_Krantz_2013} imply
\begin{enumerate}
\item The function $z=\Psi(x,y)=\Psi(u)$, $u\in\us\times\CY$, is smooth. 
\item The equation $w=\Phi(u,z)$ can be solved for $z$ and the solution $z(w,u)$ is smooth on the domain $[-\de,\de]\times(\us\times\CY)$, where $\de>0$ is sufficiently small.
\end{enumerate}
In the proof of the second statement we patch together local solutions similarly to \cite[\S 8, exercise 14]{GuillPoll74}.

Define the set
\be\label{Lambda def}
\Lambda:=\{(u,z)\in\us\times\vs:\ z=z(w,u),|w|<\de\}.
\ee
Thus $\Lambda$ is a small neighborhood of the set $\{(u,z)\in\us\times\vs:\ z=\Psi(u)\}$.
Further, define the $N\times N$ matrix function 
\be\label{M def}
Q(u):=\pa_z\Phi(u,z)|_{z=\Psi(u)},\ u\in\us\times\CY.
\ee
By assumption~\ref{ass:Phi}(3), $|\text{det}Q(u)|\ge c>0$ if $u\in\us\times\CY$. Differentiating the identity $\Phi(u,z(w,u))\equiv w$ with respect to $w$ we see that 
\be\label{zwu}\bs
z(w,u)= & \Psi(u)+Q^{-1}(u) w+O(|w|^2),\
|w|\le \de,\ u\in \us\times\CY.
\end{split}
\ee
We suppose that $\de>0$ (and the set $\Lambda$) is sufficiently small so that 
\be\label{small 2nd}
\Vert \pa_w^m z(w,u)\Vert\le c_m,\ |\text{det}(\pa_w z(w,u))|\ge c,\ m\in\BN_0^N,|w|\le\de,\ u\in \us\times\CY,
\ee
where $\Vert\cdot\Vert$ denotes any matrix norm.

One is given discrete data
\be\label{data}\bs
g(y_j,z_k),\ y_j=& \e_y j\in\CY,j\in\BZ^{n-N},\ z_k=\e k\in \CZ,k\in\BZ^N,\\ 
\e/\e_y\equiv & \text{const}.
\end{split}
\ee
Realistic data are usually represented as the sum of a useful signal (the RT of some function) and noise. The reconstruction operator $\CA$ is linear, so we assume that the useful signal is zero and the data consist only of noise: $g(y_j,z_k)=\eta_{j,k}$. 

\begin{assumptions}[Properties of noise $\eta_{j,k}$]\label{ass:noise}
$\hspace{1cm}$ 
\begin{enumerate}
\item $\eta_{j,k}$ are independent (but not necessarily identically distributed) random variables, 
\item One has
\be\label{noise var}
\Eb\eta_{j,k}^2=\e^{2\ga}\e_y^{-(n-N)}\sigma^2(y_j,z_k),\
\Eb|\eta_{j,k}|^3=o(\e^{3\ga}\e_y^{-2(n-N)}),
\ee
where $\sigma(y,z)$ is a Lipschitz continuous, bounded function on $\vs$, and the small-$o$ term is uniform in $j,k$ as $\e\to0$. 
\end{enumerate}
\end{assumptions}

Reconstruction is computed using the well-known filtered backprojection scheme.
\begin{enumerate}
\item Interpolate and, possibly, smooth the data along the $z$ variable:
\be\label{interp}
g_{\text{cont}}(y,z):=\sum_{z_k\in\CZ} \ik\bigg(\frac{z-z_k}\e\bigg) g(y,z_k),\ y=y_j\in\CY.
\ee
Here, $\ik$ is a compactly supported and sufficiently smooth function, which describes the effects of numerical interpolation and smoothing of the data. 
\item Filtering step. Compute
\be\label{flt step}\bs
N_\e^{(1)}(x,y)=&(2\pi)^{-N}\int_{\br^N}\int_{\CZ} a(x,(y,z),\xi) e^{i\xi\cdot\Phi(x,(y,z))}g_{\text{cont}}(y,z)\dd z\dd\xi,\\ 
y=&y_j\in\CY.
\end{split}
\ee
\item Backprojection step. Compute the Riemann sum
\be\label{bp step}
\nrec(x)=\e_y^{n-N}\sum_{y_j\in\CY}N_\e^{(1)}(x,y_j).
\ee
\end{enumerate}

For a domain $U\subset\br^n$, $n=1,2,\dots$, the notation $f\in C^k(\bar U)$ means that $f(x)$ is $k$ times differentiable up to the boundary and $\pa_x^m f \in L^\infty(U)$ for any multi-index $m$, $|m|\le k$. The notation $f\in C_0^k(U)$ means that $f\in C^k(\bar U)$ and $\text{supp}(f)\subset U$.

\begin{assumptions}[Properties of the kernel $\ik$]\label{ass:ker}
$\hspace{1cm}$ 
\begin{enumerate}
\item $\ik\in C_0^M(\CZ)$ for some integer $M>\max(N+\ga+1,n/2)$.
\end{enumerate}
\end{assumptions}

Assumptions~\ref{ass:grt}(1), \ref{ass:Phi}(3), and \ref{ass:ker}(1) imply that the oscillatory integral \eqref{flt step} is well-defined.

We need the following technical definition.

\begin{definition}\label{def:box-dim}\cite[Chapter 2]{Falc2014} Let $F$ be any non-empty bounded subset of $\br^n$ and let $N_\de(F)$ be the smallest
number of sets of diameter at most $\de$ which can cover $F$. The {\it upper box-counting dimension} of $F$ is defined as
\be\label{upper box def}
\overline{\text{dim}}_B(F):= \limsup_{\de\to0} \frac{\log(N_\de(F))}{-\log(\de)}.
\ee
\end{definition}

An easier, but equivalent, definition is obtained by replacing $N_\de(F)$ in \eqref{upper box def} with $N_\de^\prime(F)$, which is the number of cubes of the form $\de[m,m+\vec 1]$, $m\in\BZ^n$, intersecting $F$ \cite[Chapter 2]{Falc2014}. Here $\vec 1=(1,\dots,1)\in\BZ^n$.

Later on, we consider reconstruction in an $\e$-neighborhood of a point $x_0$, which satisfies the following properties:
\begin{assumptions}[Properties of $x_0$]\label{ass:x0} 
\hspace{1cm}
\begin{enumerate} 
\item For each $\xi\in\br^N\setminus 0$, the set 
\be\label{Ydef}
Y_1(x_0,\xi):=\{y\in\CY:\text{det}(\pa_y^2(\xi\cdot\Psi(x_0,y)))=0\}
\ee
has the upper box-counting dimension $\overline{\text{dim}}_B(Y_1(x_0,\xi))<n-N$. 
\item There exist an open set $V\in\CY$ and an open cone $\Xi\subset\br^N\setminus 0$ such that
\be\bs
&a_0(x_0,(y,\Psi(x_0,y)),\xi)\not=0,\ \forall y\in V,\xi\in\Xi,\\
&\sigma(y,\Psi(x_0,y))\not=0,\  \forall y\in V.
\end{split}
\ee
\item For any $\chx\in\br^n \setminus 0$, the set 
\be\label{Bolker}
Y_2(x_0,\chx):=\{y\in \CY: \pa_x\Psi(x_0,y) \chx=0\}
\ee
has Lebesgue measure zero.
\end{enumerate}
\end{assumptions}


For all practical purposes, we can think of Assumption~\ref{ass:x0}(1) as saying that for each $\xi\in\br^N\setminus 0$, the set of points $y\in\CY$ such that the Hessian of the function $y\to\xi\cdot\Psi(x_0,y)$ is degenerate is locally a submanifold of codimension $\ge 1$. This assumption makes sense because the solution set of a generic scalar equation $f(y)=0$, $y\in\CY$, is a surface of codimension 1.

Assumption~\ref{ass:x0}(3) can be viewed as a consequence of a local Bolker condition. In terms of $\Psi$, the incidence relation defined by the generalized RT $\R$ can be written in the form $z-\Psi(x,y)=0$. In this case, the conventional global Bolker condition becomes $\pa_y\Psi(x_1,y)\not=\pa_y\Psi(x_2,y)$ for any $x_1,x_2\in\us$, $x_1\not=x_2$, and $y\in\CY$. Localizing to a neighborhood of some $x_0\in\us$, we get
\be\label{aux_bolk}
\pa_y[\Psi(x_0+t\chx,y)-\Psi(x_0,y)]\not=0,\ 0<|t|\ll1.
\ee
Dividing by $t$ and taking the limit as $t\to0$ suggests the local condition
\be\label{Bolker alt}
\pa_y \pa_x\Psi(x_0,y)\chx\not=0\ \forall \chx\in\br^n\setminus0,y\in\CY.
\ee
Note that the first zero in \eqref{Bolker alt} stands for a $N\times(n-N)$ zero matrix. If \eqref{Bolker alt} holds, then \eqref{Bolker} holds as well. Equations \eqref{aux_bolk} and \eqref{Bolker alt} explain the intuition on which the condition \eqref{Bolker} is based. 

\section{Main results}\label{sec:main_res}

\begin{theorem} \label{thm:Lyapunov1d} Let $x_0\in\us$, $\chx\in\br^n$ be two fixed points. Suppose the domains $\us,\vs$ satisfy Assumption~\ref{ass:domains}, the operator $\CA$ satisfies Assumption~\ref{ass:grt} and Assumption~\ref{ass:Phi}, the random variables $\eta_{k,j}$ satisfy Assumption \ref{ass:noise}, the kernel $\ik$ satisfies Assumption \ref{ass:ker}, and the point $x_0$ satisfies Assumption~\ref{ass:x0}. One has
\be\label{main lim}
N^{\text{rec}}:=\lim_{\e\to0}\nrec(x_0+\e\chx)
\ee
is a Gaussian random variable, and the limit is in the sense of distributions.
\end{theorem}

Let us choose $L\ge1$ points $\chx_1,\dots,\chx_L\in\br^n$. The corresponding reconstructed vector is $\vec N_\e^{\text{rec}}:=(\nrec(x_0+\e\chx_1),\dots,\nrec(x_0+\e\chx_L))\in\br^L$. For simplicity, the dependence of $N^{\text{rec}}$ and $\vec N_\e^{\text{rec}}$ on $x_0$ and $\chx$ is suppressed from notation. 

\begin{theorem} \label{thm:Lyapunov_nd} Let $x_0\in\us$ and $\chx_l\in\br^n$, $l=1,2,\dots,L$, be fixed. Suppose the assumptions of Theorem~\ref{thm:Lyapunov1d} are satisfied. One has
\be\label{main lim vec}
\vec N^{\text{rec}}:=\lim_{\e\to0}\vec N_\e^{\text{rec}}(x_0+\e\chx)
\ee
is a Gaussian random vector, and the limit is in the sense of distributions.
\end{theorem}

Clearly, Theorem~\ref{thm:Lyapunov_nd} contains Theorem~\ref{thm:Lyapunov1d}. We decided to separately state Theorem~\ref{thm:Lyapunov1d} because its proof is easier. The proof of Theorem~\ref{thm:Lyapunov_nd} follows similar logic to Theorem~\ref{thm:Lyapunov1d}, so we discuss only the main new points of the former.

Next we remind the reader the definition of a random field (also known as a random function or topological space-valued random variable).

\begin{definition}\cite[p. 182]{Khoshnevisan2002}
Let $T$ be a topological space, endowed with its Borel field $\CB(T)$. A $T$-valued random variable $X$ on the probability space $(\Omega,\CG,\BP)$ is a measurable map $X:\Omega\to T$. In other words, for all $E\in\CB(T)$, $\{\om\in\Omega:X(\om)\in E\}\in\CG$.
\end{definition}

Let $D\subset R^n$ be a bounded domain. 
Define $C:=C(\bar D,\br)$ to be the collection of all functions continuous up to the boundary $f:\bar D\to\br$ metrized by
\be\label{Cmetric}
d(f,g)=\max_{\chx\in D}|f(\chx)-g(\chx)|,\ f,g\in C.
\ee

Recall that $G(x)$, $x\in \bar D$, is a GRF if $(G(x_1),\cdots,G(x_L))$ is a Gaussian random vector for any $L\ge1$ and any collection of points $x_1,\cdots,x_L\in \bar D$ \cite[Section 1.7]{AdlerGeomRF2010}. As is known, a GRF is completely characterized by its mean function $m(x)=\Eb G(x)$, $x\in D$ and its covariance function $\text{Cov}(x,y)=\Eb (G(x)-m(x))(G(y)-m(y))$, $x,y\in D$ \cite[Section 1.7]{AdlerGeomRF2010}. Thus, Theorem~\ref{thm:Lyapunov_nd} implies that $N^{\text{rec}}(\chx)$, $\chx\in \bar D$, is a GRF. 
For simplicity, we drop the dependence of $N^{\text{rec}}$ on $x_0$ from notation.

In the next theorem, we show that $N_\e^{\text{rec}}(x_0+\e\check x)\to N^{\text{rec}}(\chx)$, $\chx \in \bar D$, as $\e\to0$ weakly, i.e. in distribution as $C$-valued random variables (\cite[p. 185]{Khoshnevisan2002}). Recall that $M$ controls the smoothness of $\ik$ (see Assumption~\ref{ass:ker}).

\begin{theorem}\label{GRF_thm}
Let $D$ be a bounded domain with a Lipschitz boundary. Suppose the assumptions of Theorem~\ref{thm:Lyapunov1d} hold and $M>n/2$. Then $N_{\e}^{\text{rec}}(x_0+\e \chx)\to N^{\text{rec}}(\chx)$, $\chx\in \bar D$, $\e\to0$, as GRFs in the sense of weak convergence. Furthermore, $N^{\text{rec}}(\chx)$, $\chx\in \bar D$, is a GRF with zero mean and covariance
\be\label{Cov main}\begin{split}
\text{Cov}(\chx,\chy)
=&C(\chx-\chy),\\
C(\vartheta):=&\int_{\CY}(G\star G)\big(y,\pa_x\Psi(y)\cdot \vartheta\big)\sigma^2(y,\Psi(y))\dd y,\\
(G\star G)(y,\vartheta):=&\int_{\br^N}G(y,\vartheta+r)G(y,r)\dd r,
\end{split}
\ee
and sample paths of $N^{\text{rec}}(\chx)$ are continuous with probability $1$.
\end{theorem}

\section{Examples}\label{sec:examples}

\subsection{Classical Radon transform in $\br^2$ and $\br^3$}\label{ssec:CRT}
We begin with the classical RT in $\br^2$. The following conventional data parametrization is used:
\be
y=\al,\ z=p,\ \Phi(x,\al,p)=p-\vec\al\cdot x,\ \Psi(x,\al)=\vec\al\cdot x,
\ee
where $\vec\al=(\cos\al,\sin\al)$. Then $\pa_p \Phi(x,\al,p)=1$, so Assumption~\ref{ass:Phi}(3) is satisfied. From \eqref{Ydef}, 
$Y_1(x_0,\xi)=\{\al\in[0,2\pi):\vec\al\cdot x_0=0\}$, which is a set consisting of two points, i.e. it has the upper box-counting dimension of zero, as long as $x_0\not=0$. Also, $Y_2(x_0,\chx)=\{\al\in[0,2\pi):\vec\al\cdot \check x=0\}$ (cf. \eqref{Bolker}), which is a set of measure zero. Thus, Assumptions~\ref{ass:x0}(1,3) are satisfied.

Next consider the classical RT in $\br^3$. In this case we use the conventional parametrization
\be\bs
&y=(\al,\theta),\ p=z,\ \Phi(x,\al,\theta,p)=p-\vec\Theta\cdot x,\ \Psi(x,\al,\theta)=\vec\Theta\cdot x,\\
&\vec\Theta:=(\cos\al\cos\theta,\cos\al\sin\theta,\sin\al),\ \theta\in[0,2\pi),|\al|<\pi/2.
\end{split}
\ee
Then $\pa_p \Phi(x,\al,\theta,p)=1$, so Assumption~\ref{ass:Phi}(3) is satisfied. By \eqref{Ydef}, we compute the Hessian:
\be\bs
&\pa_{(\al,\theta)}^2\Psi=-\begin{pmatrix} (\vec\theta\cdot\hat x)\cos\al +x_3\sin\al  & (\vec\theta^\perp\cdot\hat x)\sin\al\\ (\vec\theta^\perp\cdot\hat x)\sin\al & (\vec\theta\cdot\hat x)\cos\al
\end{pmatrix},\\
&\vec\theta=(\cos\theta,\sin\theta),\ \vec\theta^\perp=(-\sin\theta,\cos\theta),\ 
\hat x=(x_1,x_2),\ x=(\hat x,x_3).
\end{split}
\ee
Here we have assumed without loss of generality that $\xi=1$. Clearly, $\text{det}(\pa_{(\al,\theta)}^2\Psi)\equiv 0$ if $\hat x=0$. Thus, all points $x=(0,0,x_3)$, $x_3\in\br$, are exceptional, because they violate Assumption~\ref{ass:x0}(1). Suppose now $\hat x\not=0$. After simple transformations we obtain
\be
2\text{det}(\pa_{(\al,\theta)}^2\Psi)=|\hat x|^2\cos(2\al)+x_3(\vec\theta\cdot\hat x)\sin(2\al)+[2(\vec\theta\cdot\hat x)^2-|\hat x|^2].
\ee
For each $\theta\in[0,2\pi)$, we can find at most finitely many solutions $|\al|<\pi/2$ to $\text{det}(\pa_{(\al,\theta)}^2\Psi)=0$ and, generically, they depend smoothly on $\theta$. Hence, the sets $Y_1(x,\xi)$, $\xi\not=0$, have the upper box-counting dimension of one, as long as $\hat x\not=0$. For the cone beam transform, $n=3$ and $N=1$, so Assumption~\ref{ass:x0}(1) is satisfied.

Finally, $Y_2(x,\chx)=\{\vec\Theta\in\mathcal S^2:\vec\Theta\cdot \check x=0\}$ (cf. \eqref{Bolker}), which is  a set of measure zero in the unit sphere. Thus, Assumption~\ref{ass:x0}(3) is satisfied as well.

\subsection{Cone beam transform in $\br^3$}\label{ssec:CBT}
Our next example is the cone beam transform in $\br^3$ with a circular source trajectory. The detector coordinates are $(u,v)\in\br^2$ (the analog of $z$), the source coordinate is $s\in[0,2\pi)$ (the analog of $y$), see Figure~\ref{fig:cbct}, and $n=3$, $N=2$. The source trajectory is given by
\be\label{source}
P(s)=(R\cos s,R\sin s,0),\ 0\le s<2\pi.
\ee
The (virtual) detector is flat, passes through the origin, and rotates together with the source. Points on the detector are parametrized as follows:
\be\label{detector}
Z(s,u,v)=u(-\sin s,\cos s,0)+v(0,0,1).
\ee
For a point $x=(x_1,x_2,x_3)\in\br^3$, its stereographic projection from the source $P(s)$ to a point $(u,v)=(U(x,s),V(x,s))$ in the detector plane is given by
\be\label{det proj}\bs
&U(x,s)=T(x,s)(-x_1\sin s+x_2\cos s),\ V(x,s)=T(x,s)x_3,\\ 
&T(x,s)=[1-(x_1\cos s+x_2\sin s)/R]^{-1}.
\end{split}
\ee
The support of $f$ and reconstruction domain $\us$ are contained inside the cylinder $x_1^2+x_2^2\le c<R$. The two key functions are given by
\be\label{CB fns}
\Phi(x,s,u,v)=(u-U(x,s),v-V(x,s)),\ \Psi(x,s)=(U(x,s),V(x,s)).
\ee
Clearly, $\pa_{(u,v)}\Phi$ is a $2\times2$ identity matrix, so Assumption~\ref{ass:Phi}(3) is satisfied. 

\begin{figure}[h]
{\centerline{
{\epsfig{file={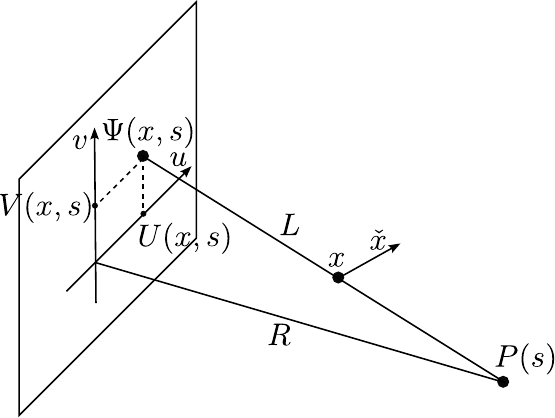}, width=8cm}}
}}
\caption{Illustration of cone beam geometry.}
\label{fig:cbct}
\end{figure}

From \eqref{Ydef}, $Y_1(x,\xi)=\{s\in[0,2\pi):\pa_s^2(\xi\cdot \Psi(x,s)=0\}$, $\xi\in\br^2\setminus0$. Hence, generally, Assumption~\ref{ass:x0}(1) is violated if the projection of $x$ onto the detector contains a straight line segment, see Figure~\ref{fig:badcase1}, left panel. For the circular source trajectory \eqref{source}, this is the case when $x_3=0$, i.e. when $x$ is in the plane of the source trajectory, see Figure~\ref{fig:badcase1}, right panel. As is easy to see, the projection of $x$ to the detector is an ellipse as long as $x_3\not=0$. Indeed, from \eqref{det proj} we find
\be\bs
&(x_2/x_3)\cos s-(x_1/x_3)\sin s=U/V,\\
&(x_1/R)\cos s+(x_2/R)\sin s=1-(x_3/V).
\end{split}
\ee
Eliminating $\sin s$ and $\cos s$ gives 
\be
(x_1^2+x_2^2)V^2=x_3^2U^2+R^2(V-x_3)^2.
\ee
This is an ellipse, because the reconstruction domain is inside the source trajectory, i.e. $x_1^2+x_2^2<R^2$. 
Therefore Assumption~\ref{ass:x0}(1) is satisfied if $x_3\not=0$.

\begin{figure}[h]
{\centerline{
{\epsfig{file={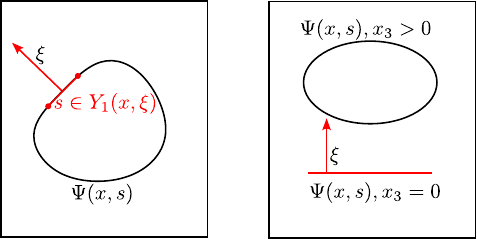}, width=11cm}}
}}
\caption{Illustration of the case when the set $Y_1(x,\chx)$ is an interval for a general source trajectory (left panel). The right panel illustrates the case of a circular source trajectory.}
\label{fig:badcase1}
\end{figure}

The meaning of Assumption~\ref{ass:x0}(3) is best understood using Figures~\ref{fig:cbct} and \ref{fig:badcase2}. We have $Y_2(x,\chx)=\{s\in[0,2\pi):\pa_t\Psi(x+t\chx,s)|_{t=0}=0\}$, where $\Psi(x,s)$ is the projection of $x$ to the detector. It is clear from the figures that the directional derivative is zero only when $\chx$ is along the line $L(x,s)$ through $x$ and the source $P(s)$. This implies that Assumption~\ref{ass:x0}(3) is violated only if there exists a line $L$ through $x$ such that the intersection of $L$ with the source trajectory contains a line segment as illustrated in Figure~\ref{fig:badcase2}. Clearly, for a circular trajectory this does not happen.

\begin{figure}[h]
{\centerline{
{\epsfig{file={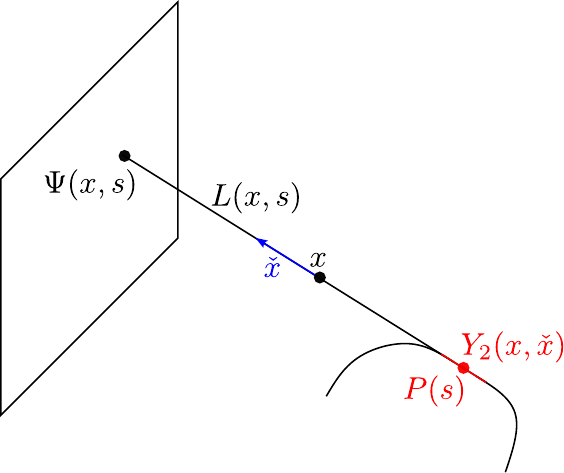}, width=8cm}}
}}
\caption{Illustration of the case when the set $Y_2(x,\chx)$ has positive measure. The values $s\in Y_2(x,\chx)$ parametrize a line segment (shown in red), which is a subset of the source trajectory. The vector $\check x$ (shown in blue) is parallel to the line segment. The lines $L(x,s)$, $s\in Y_2(x,\chx)$, are all the same and contain the segment.}
\label{fig:badcase2}
\end{figure}

\section{Proof of Theorem~\ref{thm:Lyapunov1d}: Contribution of the principal symbol}\label{sec:prf beg}

Throughout the proof we make the additional assumption
\be\label{extra cond}
a(x,(y,z),\xi)=a_0(x,(y,z),\xi)\equiv0,\ (x,(y,z))\not\in\Lambda.
\ee
This assumption does not affect the validity of Theorem~\ref{thm:Lyapunov1d}. One can always modify $\CA$ by a smoothing operator so that this condition holds \cite[Theorem 2.2, Ch. VI]{trev2}. See also the last paragraph in section~\ref{sec:LOTs}.

In view of \eqref{interp}, \eqref{flt step}, consider the integral
\be\label{F def}
F_\e(u,z^\prime):=(2\pi)^{-N}\int_{\br^N}\int_{\CZ} a_0(u,z,\xi)\ik\bigg(\frac{z-z^\prime}\e\bigg)e^{i\xi\cdot\Phi(u,z)}\dd z \dd\xi.
\ee

\subsection{An upper bound for $|F_\e(u,z^\prime)|$}
Change variable $z\to w=\Phi(u,z)$ and then $w\to \check w=w/\e$, $\xi\to\hat\xi=\e\xi$:
\be\label{F v2}\begin{split}
F_\e(u,z^\prime)=&(2\pi)^{-N}\int_{\br^N}\int_{|\check w|\le\de/\e} \frac{a_0(u,z,\xi)}{|\text{det}(\pa_z \Phi(u,z))|}\ik\bigg(\frac{z-z^\prime}\e\bigg)\e^Ne^{i\e\xi\cdot \check w}\dd \check w \dd\xi\\
=&\e^{-\ga}(2\pi)^{-N}\int_{\br^N}\int_{W_\e(u,z^\prime)} a_1(u,z,\hat\xi)\ik\bigg(\frac{z-z^\prime }\e\bigg)e^{i\hat\xi\cdot \check w}\dd \check w \dd\hat\xi,\\
a_1(u,z,\xi):=&a_0(u,z,\xi)|\text{det}(\pa_z \Phi(u,z))|^{-1},\ z=z(\e\check w,\e).
\end{split}
\ee
By \eqref{small 2nd}, $a_1$ is well-defined. Clearly, the matrices $\pa_z \Phi(u,z)$, where $z=z(w,u)$, and $\pa_w z(w,u)$ are the inverses of each other. Here and throughout this section we set 
\be\label{chech v def}
\check v:=(z^\prime-\Psi(u))/\e,\ W_\e(u,z^\prime):=\{\check w\in\br^N: (z(\e \check w,u)-z^\prime)/\e\in\text{supp}(\ik)\}.
\ee
Since $\ik$ is compactly supported, \eqref{small 2nd} and Assumption~\ref{ass:Phi}(2) imply:
\be\label{diam W}\bs
&\text{diam}(W_\e(u,z^\prime))\le c, (u,z^\prime)\in\Lambda,0<\e\ll1;\\
&\check w\in W_\e(u,z^\prime),(u,z^\prime)\in\Lambda\text{ implies }\big[Q^{-1}(u)+O(\e|\check w|)\big]\check w\in B(\check v,c).
\end{split}
\ee
Here $B(\check v,c)$ is a ball with center $\check v$ and radius $c$.

\noindent
{\bf Case 1: $|\check v|\le c$.} Since $\ik$ is compactly supported, \eqref{small 2nd} implies that the domain of integration with respect to $\check w$ in \eqref{F v2}, i.e. the set $W_\e(u,z^\prime)$, is a subset of a disk $|\check w|\le c$. Therefore,
\be\label{F v3}\begin{split}
\e^\ga F_\e(u,z^\prime)=&(2\pi)^{-N}\int_{\br^N} H_\e(u,z^\prime,\hat\xi)\dd\hat\xi,\
(u,z^\prime )\in\Lambda,\\
H_\e(u,z^\prime,\hat\xi):=&\int_{|\check w|\le c} a_1(u,z(\e \check w,u),\hat\xi)\ik\bigg(\frac{z(\e \check w,u)-z^\prime }\e\bigg)e^{i\hat\xi\cdot \check w}\dd \check w.
\end{split}
\ee

Let $M_1$ be the smallest integer that satisfies $M_1>N+\ga$. Repeated integration by parts with respect to $\check w$ gives
\be\label{F v4}\begin{split}
&H_\e(u,z^\prime ,\hat\xi)=
\int_{|\check w|\le c} (L^T)^{M_1}\bigg[a_1(u,z,\hat\xi)\ik\bigg(\frac{z-z^\prime }\e\bigg)\bigg]e^{i\hat\xi\cdot \check w}\dd \check w,\\
&L:=(1+|\hat\xi|^2)^{-1}(1-i\hat\xi\cdot\pa_{\check w}),\ z=z(\e \check w,u),\ (u,z^\prime )\in\Lambda.
\end{split}
\ee
In this integration by parts we use Assumption~\ref{ass:grt}(2) and Assumption~\ref{ass:ker}. This implies
\be\label{G1 bnd}\begin{split}
|H_\e(u,z^\prime ,\hat\xi)|\le & c (1+|\hat\xi|)^{\ga-M_1},\ \hat\xi\in\br^N,\
(u,z^\prime )\in\Lambda.
\end{split}
\ee
The integral on the first line in \eqref{F v3} is absolutely convergent because $M_1>N+\ga$, hence 
\be\label{F1 bnd}
\e^\ga |F_\e(u,z^\prime )|\le c,\ \text{if }|\check v|\le c,(u,z^\prime )\in\Lambda.
\ee

\noindent
{\bf Case 2: $|\check v|\ge c$.}
Let $A_1\in\s^\prime(\br^N)$ be the distribution
\be\label{A1 ker}
A_1(u,z,\vartheta):=(2\pi)^{-N}\int a_1(u,z,\hat\xi)e^{i\hat\xi\cdot \vartheta}\dd\hat\xi,\
(u,z)\in\Lambda,\ \vartheta\in\br^N.
\ee
We view $\vartheta$ as the argument of $A_1$, and $u$ and $z$ -- as parameters. As is known, Assumption~\ref{ass:grt}(2) implies $A_1\in C^\infty(\br^N\setminus 0)$ and $A_1$ is homogeneous of degree $-(N+\ga)$ \cite[Theorems 7.1.16 and 7.1.18]{hor}:
\be\label{A1 homog}
A_1(u,z,t\vartheta)\equiv t^{-(N+\ga)}A_1(u,z,\vartheta),\ t>0.
\ee
More generally we have $A_1\in C^\infty(\us\times\vs\times(\br^N\setminus0))$, if $u,z,\vartheta$ are all viewed as arguments.

If $|\check v|$ is sufficiently large and $\de>0$ is sufficiently small (cf. \eqref{Lambda def}, \eqref{zwu} and \eqref{small 2nd}), the set $W_\e(u,z^\prime)$ is bounded away from $\check w=0$. 
More precisely, $W_\e(u,z^\prime)\subset \{\check w\in\br^N:\ |\check w|\ge c|\check v|\}$ for some $c>0$ independent of $0<\e\ll1$ and $(u,z^\prime)\in\Lambda$. Combining with \eqref{diam W} this gives
\be\label{part 2}\begin{split}
\e^\ga F_\e(u,z^\prime )=&\int_{W_\e(u,z^\prime)} A_1(u,z^\prime ,\check w)\ik\bigg(\frac{z(\e \check w,u)-z^\prime }\e\bigg)\dd \check w,\ (u,z^\prime )\in\Lambda,
\end{split}
\ee
and \eqref{A1 homog} implies
\be\label{F2 bnd}\begin{split}
\e^\ga |F_\e(u,z^\prime )|\le c (1+|\check v|)^{-(N+\ga)}\text{  if  }  |\check v|\ge c,\
(u,z^\prime )\in\Lambda.
\end{split}
\ee
Combining \eqref{F1 bnd} and \eqref{F2 bnd} gives
\be\label{Fall bnd}\begin{split}
\e^\ga |F_\e(u,z^\prime )|\le c (1+|\check v|)^{-(N+\ga)},\ 
(u,z^\prime )\in\Lambda.
\end{split}
\ee

\subsection{Approximation of $F_\e$ by simpler functions}
Replace $z=z(\e\check w,u)$ with $z=\Psi(u)$ in the argument of $a_1$ in \eqref{F v2} and define
\be\label{F1 def}\begin{split}
F_\e^{(1)}(u,z^\prime ):=&\e^{-\ga} (2\pi)^{-N}\iint a_1(u,\Psi(u),\hat\xi)\ik\bigg(\frac{z(\e\check w,u)-z^\prime }\e\bigg)e^{i\hat\xi\cdot \check w}\dd \check w \dd\hat\xi.
\end{split}
\ee
Since $\ik$ is compactly supported, 
\be\label{triangle}
|z(\e\check w,u)-\Psi(u)|\le |z(\e\check w,u)-z^\prime|+|z^\prime-\Psi(u)|\le c\e+\e|\check v|,
\ee
cf. \eqref{chech v def} and \eqref{diam W}. The following inequality is proven in section~\ref{ssec:ff1}.
\be\label{F-F1}\begin{split}
&\e^{\ga}|F_\e(u,z^\prime )-F_\e^{(1)}(u,z^\prime )|\le c\e (1+|\check v|)^{-(N+\ga-1)},\ (u,z^\prime)\in\Lambda.
\end{split}
\ee


Next we replace $z(\e\check w,u)$ in the arguments of $\ik$ in \eqref{F1 def} and define
\be\label{barF def}\begin{split}
\bar F_\e(u,z^\prime ):=&\e^{-\ga} (2\pi)^{-N}\iint a_1(u,\Psi(u),\hat\xi)\ik\bigg(\frac{\bar z(\e\check w,u)-z^\prime }\e\bigg)e^{i\hat\xi\cdot \check w}\dd \check w \dd\hat\xi,\\
\bar z(w,u):=&\Psi(u)+Q^{-1}(u)w.
\end{split}
\ee
Since $\ik$ and its derivatives are bounded, \eqref{small 2nd} implies 
\be\label{del phi}\bs
&\bigg|\pa_{\check w}^m\bigg[\ik\bigg(\frac{z(\e\check w,u)-z^\prime }\e\bigg)-\ik\bigg(\frac{\bar z(\e\check w,u)-z^\prime }\e\bigg)\bigg]\bigg|\\
&\le c\big[\min(1,\e|\check w|^2)+\e|\check w|\big],\ |m|\le M_1.
\end{split}
\ee
This is where we use that $\ik$ has extra smoothness $\ik\in C_0^{M_1+1}$ (compared with the requirement $\ik\in C_0^{M_1}$, which suffices for \eqref{F v4} to hold). 
The following claim is proven in section~\ref{ssec:f1fb}.
\be\label{F1-Fbar}\begin{split}
\e^{\ga}|F_\e^{(1)}(u,z^\prime )-\bar F_\e(u,z^\prime )|\le & c\big[\min(1,\e|\check v|^2)+\e|\check v|\big] (1+|\check v|)^{-(N+\ga)},\\ 
(u,z^\prime)\in & \Lambda.
\end{split}
\ee

Combining \eqref{F-F1} and \eqref{F1-Fbar} yields
\be\label{F-Fbar}\begin{split}
&\e^{\ga}|F_\e(u,z^\prime )-\bar F_\e(u,z^\prime )|\le c\frac{\e(1+|\check v|)+\min(1,\e|\check v|^2)}{(1+|\check v|)^{N+\ga}},\ (u,z^\prime)\in \Lambda.
\end{split}
\ee
After simple transformations
\be\label{bFG}\begin{split}
\e^{\ga} \bar F_\e(u,z^\prime )=&G(u,\check v),\
G(u,\vartheta):=(2\pi)^{-N}\int_{\br^N} a_2(u,\hat\mu)\tilde\ik(\hat\mu)e^{i\hat\mu\cdot \vartheta}\dd \hat\mu,\\
a_2(u,\hat\mu):=&a_1(u,\Psi(u),Q^{-T}(u)\hat\mu),
\end{split}
\ee
where $\tilde\ik$ is the Fourier transform of $\ik$. As is easily seen, $a_2$ satisfies \eqref{a0_props} with the variable $z$ and set $\CZ$ omitted. Recall that $\vs=\CY\times\CZ$. The main property of the symbol $a_2$ is that it is independent of $z^\prime$.
Similarly to \eqref{Fall bnd},
\be\label{G bnd}
|G(u,\vartheta)|\le c (1+|\vartheta|)^{-(N+\ga)},\ 
u\in\us\times\CY,\vartheta\in\br^N.
\ee
Arguing similarly to \eqref{F v3} -- \eqref{Fall bnd}, the assumption $\ik\in C_0^{M_1+1}(\br^n)$ implies that $G\in C^1(\us\times\CY\times\br^N)$ and
\be\label{Gders bnds}\bs
|\pa_u G(u,\vartheta)|&\le c (1+|\vartheta|)^{-(N+\ga)},\ |\pa_{\vartheta}G(u,\vartheta)|\le c (1+|\vartheta|)^{-(N+\ga+1)},\\ 
u&\in\us\times\CY,\vartheta\in\br^N.
\end{split}
\ee
The second inequality follows because the amplitudes 
$$
\hat\mu_l a_2(u,\hat\mu),\ 1\le l\le N, 
$$
are homogeneous in $\hat\mu$ of degree $\ga+1$.

\section{Proof of Theorem~\ref{thm:Lyapunov1d}: contribution of the lower order terms of the FIO $\CA$}
\label{sec:LOTs}

Using \eqref{F v2} denote
\be\label{rem symb}\begin{split}
\Delta a_1(u,z,\xi):=&(a(u,z,\xi)-a_0(u,z,\xi))|\pa_z \Phi(u,z)|^{-1},
\end{split}
\ee
and
\be\label{delF v3}\begin{split}
\e^\ga \Delta F_\e(u,z^\prime ):=&(2\pi)^{-N}\int_{\br^N} \Delta H_\e(u,z^\prime ,\hat\xi)\dd\hat\xi,\\
\Delta H_\e(u,z^\prime ,\hat\xi):=&\int_{W_\e(u,z^\prime)} \e^\ga\Delta a_1(u,z(\e \check w,u),\hat\xi/\e)\\&\hspace{1cm}\times\ik\bigg(\frac{z(\e \check w,u)-z^\prime }\e\bigg)e^{i\hat\xi\cdot \check w}\dd \check w.
\end{split}
\ee

{\bf Case 1: $|\check v|\le c$.}
Using the same operator $L$ as in \eqref{F v4}, repeated integration by parts with respect to $\check w$ gives
\be\label{delF v4}\begin{split}
\Delta H_\e(u,z^\prime ,\hat\xi)=&
\int_{|\check w|\le c} (L^T)^{M_1}\bigg[\e^\ga\Delta a_1(u,z,\hat\xi/\e)\ik\bigg(\frac{z-z^\prime }\e\bigg)\bigg]e^{i\hat\xi\cdot \check w}\dd \check w,\\
z=&z(\e \check w,u),
\end{split}
\ee
This, Assumption~\ref{ass:grt}(3), and Assumption~\ref{ass:ker} imply
\be\label{delG1 bnd}\begin{split}
|\Delta H_\e(u,z^\prime ,\hat\xi)|\le & c {\e^\ga}{(1+|\hat\xi/\e|)^{\ga^\prime}} (1+|\hat\xi|)^{-M_1}\\
\le & c {\e^{\ga-\ga^\prime}} (1+|\hat\xi|)^{\ga^\prime-M_1},\ \hat\xi\in\br^N,\, (u,z^\prime )\in\Lambda.
\end{split}
\ee
The integral on the first line in \eqref{delF v3} is absolutely convergent because $M_1>N+\ga^\prime$, hence 
\be\label{delF1 bnd}
\e^\ga |\Delta F_\e(u,z^\prime )|=O(\e^{\ga-\ga^\prime}),\ |\check v|\le c,(u,z^\prime )\in\Lambda. 
\ee

{\bf Case 2: $|\check v|\ge c$.}
When $w=0$ is not in the domain of integration, we have
\be\label{delF v3_2}\begin{split}
\Delta F_\e(u,z^\prime )
=&\int_{|w|\le\de} \Delta A_1(u,z(w,u),w)\ik\bigg(\frac{z(w,u)-z^\prime }\e\bigg)\dd w,\\
\Delta A_1(u,z,\vartheta):=&(2\pi)^{-N}\int_{\br^N} \Delta a_1(u,z,\xi)e^{i\xi\cdot \vartheta} \dd\xi.
\end{split}
\ee
Since $\ga>-N/2$, we can always assume that $N+\ga^\prime>0$. By assumption~\ref{ass:grt}(3) and \cite[Theorem 5.12]{abels12},
\be\label{delA bnd}
|\Delta A_1(u,z,\vartheta)|\le c|\vartheta|^{-(N+\ga^\prime)},\ (u,z)\in\Lambda,0<|\vartheta|\le c.
\ee
The fact that the amplitude $\Delta a_1(u,z,\xi)$ is not smooth at the origin $\xi=0$ is irrelevant, since this part contributes a bounded term. We have used here that $\vartheta$ is confined to a bounded set, and the goal of \eqref{delA bnd} is to provide a bound near $\vartheta=0$. Substituting \eqref{delA bnd} into the first equation in \eqref{delF v3_2} and changing variable $w=\e\check w$ gives using \eqref{diam W} and that $W_\e(u,z^\prime)\subset \{\check w\in\br^N:\ |\check w|\ge c|\check v|\}$:
\be\label{delF est}
\e^{\ga}|\Delta F_\e(u,z^\prime)|\le c \e^{\ga-\ga^\prime}(1+|\check v|)^{-(N+\ga^\prime)},\ |\check v|\ge c,(u,z^\prime)\in\Lambda.
\ee
Combining with \eqref{delF1 bnd} yields
\be\label{delF est final}
\e^{\ga}|\Delta F_\e(u,z^\prime)|\le c \e^{\ga-\ga^\prime}(1+|\check v|)^{-(N+\ga^\prime)},\ (u,z^\prime)\in\Lambda.
\ee
Recall that, by construction, $\Delta F_\e(u,z^\prime)=0$ if $(u,z^\prime)\in(\us\times\vs)\setminus\Lambda$.

Suppose now that $K$ is a smoothing operator. Such an operator may be needed to ensure that $a(u,z,\xi)\equiv0$ if $(u,z)\not\in\Lambda$ (see \eqref{extra cond}). The analog of \eqref{F def} is
\be\label{Fsm def}
F_\e(u,z^\prime ):=\int_{\CZ} K(u,z)\ik((z-z^\prime )/\e)\dd z=O(\e^N),\ \e\to0,
\ee
uniformly in $(u,z^\prime)\in\us\times\vs$.

\section{End of proof of Theorem~\ref{thm:Lyapunov1d} and proof of Theorem~\ref{thm:Lyapunov_nd}}\label{sec:Lyapunov}
In what follows we take $x=x_0+\e\chx$. By \eqref{Gders bnds},
\be\label{G approx}\bs
G\bigg(x,y,\frac{z^\prime-\Psi(x,y)}\e\bigg)=&G\bigg(x_0,y,\frac{z^\prime-\Psi(x_0,y)}\e-\pa_x\Psi(x_0,y)\cdot\chx\bigg)\\
&+\frac{O(\e)}{(1+|\check v|)^{N+\ga}}.
\end{split}
\ee
In this section $x_0$ is fixed and dropped from most notation. 
Since $z^\prime=z_k=\e k$ (cf. \eqref{data}, \eqref{interp}, \eqref{flt step}, and \eqref{F def}), reconstruction from noise can be written as follows
\be\label{recon}\bs
\nrec(x_0+\e\chx)=&\kappa_\e\sum_{(y_j,z_k)\in\vs} \big[G\big(y_j,k-b_j\big)+R_{j,k}\big]\eta_{j,k},\\ 
b_j:=&\pa_x\Psi(x_0,y_j)\cdot\chx+(\Psi(x_0,y_j)/\e),\ \kappa_\e:=\e^{-\ga}\e_y^{n-N},
\end{split}
\ee
where $R_{j,k}$ is the remainder. By \eqref{F-Fbar}, \eqref{delF est final}, \eqref{Fsm def}, and \eqref{G approx} the remainder satisfies
\be\label{R bnd}\begin{split}
|R_{j,k}|\le & c\frac{\e(1+t)+\min(1,\e t^2)}{(1+t)^{N+\ga}}
+c \frac{\e^{\ga-\ga^\prime}}{(1+t)^{N+\ga^\prime}}+O(\e^{N+\ga}),\\ 
t=&|k-b_j|,\ (y_j,z_k)\in\vs.
\end{split}
\ee

The following result is proven in appendix~\ref{sec:generic cond prf}. Its proof requires Assumption~\ref{ass:x0}(1).

\begin{lemma}\label{lem:generic cond}
Suppose the assumptions of Theorem~\ref{thm:Lyapunov1d} are satisfied.
For any $\xi\in\br^N\setminus 0$, $u\in\BZ^{n-N}$, and any $\de>0$ sufficiently small, there exist $p>0$ and a finite collection of hypercubes $\mathcal B_k\subset \br^{n-N}$ such that
\be\label{cor_cond1_1}\begin{split}
&\{y\in \CY:\, |\pa_y(\xi\cdot\Psi(x_0,y))-u|\le p\}\subset \cup_k\mathcal B_k,\ 
\sum_k\text{Vol}(\mathcal B_k)\le\de.
\end{split}
\ee
\end{lemma}

We have the following lemma, which is proven in appendix~\ref{sec:variance}. The proof uses Lemma~\ref{lem:generic cond}.

\begin{lemma}\label{lem:Lyapunov} Suppose the assumptions of Theorem~\ref{thm:Lyapunov1d} are satisfied. With $b_j$ and $\kappa_\e$ defined in \eqref{recon}, one has
\be\label{main lim L1d}
\lim_{\e\to0}\frac{\kappa_\e^3\sum_{(y_j,z_k)\in\vs}\lvert G(y_j,k-b_j)+R_{j,k}\rvert^3\Eb\lvert \eta_{j,k}\rvert^3}{\big[\kappa_\e^2\sum_{(y_j,z_k)\in\vs}(G(y_j,k-b_j)+R_{j,k})^2\Eb\eta_{j,k}^2\big]^{3/2}}=0.
\ee
\end{lemma}

Lemma~\ref{lem:Lyapunov} implies that the family of random variables $\nrec(x_0+\e\chx)$, $\e>0$, satisfies the Lyapunov  condition for triangular arrays \cite [Definition 11.1.3]{ath_book}. By \cite[Corollary 11.1.4]{ath_book}, $N^{\text{rec}}:=\lim_{\e\to0}\nrec(x_0+\e\chx)$ is a Gaussian random variable, where the limit is in the sense of convergence in distribution. This completes the proof of Theorem~\ref{thm:Lyapunov1d}.

\subsection{Finite-dimensional distributions. Proof of Theorem~\ref{thm:Lyapunov_nd}}\label{ssec:grid recon}

Recall that the random vector $\vec N_\e^{\text{rec}}$ is defined in the paragraph preceding Theorem~\ref{thm:Lyapunov_nd}. Pick any vector $\vec\theta\in \br^L$. By \eqref{recon}
\be\label{recon dttpr 1}\begin{split}
\zeta_\e:=\vec\theta\cdot \vec N_\e^{\text{rec}} =& \kappa_\e \sum_{(y_j,z_k)\in\vs} \biggl[\sum_{l=1}^L \theta_l \big(G\big(y_j,k-b_j^{(l)}\big)+R_{j,k}^{(l)}\big)\biggr]\eta_{j,k},\\ 
b_j^{(l)}:=&\pa_x\Psi(x_0,y_j)\cdot\chx_l+(\Psi(x_0,y_j)/\e).
\end{split}
\ee 
To show that $\vec N_\e^{\text{rec}}$ converges in distribution to a Gaussian random vector, it suffices to show that for any $\vec\theta\in \br^L\setminus0$, $\lim_{\e\to0}\zeta_{\e}$ is a Gaussian random variable \cite [Theorem 10.4.5]{ath_book}. The following result is proven in appendix~\ref{sec:prf outline}.

\begin{lemma}\label{lem:Lyapunov_nd} Suppose the assumptions of Theorem~\ref{thm:Lyapunov_nd} are satisfied. With $b_j^{(l)}$ defined in \eqref{recon dttpr 1} and $\kappa_\e$ defined in \eqref{recon}, one has
\be\label{main lim Lnd}
\lim_{\e\to0}\frac{\kappa_\e^3\sum_{(y_j,z_k)\in\vs}\lvert \sum_{l=1}^L \theta_l \big(G\big(y_j,k-b_j^{(l)}\big)+R_{j,k}^{(l)}\big)\rvert^3\Eb\lvert \eta_{j,k}\rvert^3}{\big[\kappa_\e^2\sum_{(y_j,z_k)\in\vs}\big[\sum_{l=1}^L \theta_l \big(G\big(y_j,k-b_j^{(l)}\big)+R_{j,k}^{(l)}\big)\big]^2\Eb\eta_{j,k}^2\big]^{3/2}}=0.
\ee
\end{lemma}

It follows from \cite [Corollary 11.1.4]{ath_book} that $\zeta=\lim_{\e\to0} \zeta_{\e}$ is a Gaussian random variable. Hence, by \cite [Theorem 10.4.5]{ath_book}, $\lim_{\e\to0}\vec N_\e^{\text{rec}}$ is a Gaussian random vector, where as before, the limit is in the sense of convergence in distribution. This completes the proof of Theorem~\ref{thm:Lyapunov_nd}.

\section{Proof of Theorem~\ref{GRF_thm}}\label{sec:thrd thrm prf}

We use the following definition and theorem.

\begin{definition}[{\cite[p. 189]{Khoshnevisan2002}}]\label{def:kh_thm}
Let $\BP_n$ be the distribution of a $C$-valued random variable $X_n$, $1\leq n\leq \infty$. The collection $(\BP_n)$ is tight if for all $\de\in(0,1)$, there exists a compact set $\Gamma_\de\in C$ such that $\sup_n \BP(X_n\not\in\Gamma_\de)\le \de$. 
\end{definition}

\begin{theorem}[{\cite[Proposition 3.3.1]{Khoshnevisan2002}}]\label{kh_thm}
Suppose $X_n,\ 1\leq n\leq \infty$, are $C$-valued random variables. Then $X_n\to X_{\infty}$ weakly (i.e., the distribution of $X_n$ converges to that of $X_\infty$, see {\cite[p. 185]{Khoshnevisan2002}}) provided that:
\begin{enumerate}
\item Finite dimensional distributions of $X_n$ converge to that of $X_{\infty}$.
\item $(X_n)$ is a tight sequence.
\end{enumerate}
\end{theorem}

Theorem~\ref{thm:Lyapunov_nd} asserts that all finite-dimension distributions of $\nrec(x_0+\e\chx)$ converge to that of the GRF $N^{\text{rec}}(\chx)$. Thus it remains to verify property (2) of Theorem \ref{kh_thm}. To this end, we
introduce the sets 
\be\label{compact sets}
\Gamma_\de:=\{f\in C:\, \Vert f\Vert_{W^{M,2}(D)}^2 \le 1/\de\}.
\ee
Recall that $M$ controls the smoothness of $\ik$ (see Assumption~\ref{ass:ker}).
Recall also that $W^{l,p}(D)$ is the closure of $C^\infty(\bar D)$ in the norm:
\be\label{Wkp norm}
\Vert f\Vert_{l,p}:=\bigg(\int_{D}\sum_{|m|\le l}|\pa_x^m f(x)|^p\dd x\bigg)^{1/p},\ f\in C^\infty(\bar D).
\ee

Since $D$ has a Lipschitz continuous boundary and $M>n/2$, from \cite[Theorem 7.26, p. 171]{GilbTrud} it follows that the imbedding $W^{M,2}(D)\hookrightarrow C(\bar D)$ is compact. Hence the set $\Gamma_\de\subset C$ is compact for every $\de>0$. 

Adopting the argument \eqref{chech v def}--\eqref{Fall bnd} to the full symbol of $\CA$ (which only requires replacing \eqref{A1 homog} with \cite[Theorem 5.12]{abels12}), it is easy to see that
\be\label{Fall chxder bnd}\begin{split}
&\e^\ga |\pa_{\chx}^m F_\e(x,y,z)|\le c (1+|z-\Psi(x,y)|/\e)^{-(N+\ga)},\\ 
&x=x_0+\e\chx,\ \chx\in D,\ (y,z)\in\vs,\ m\in\BN_0^n,|m|\le M.
\end{split}
\ee
This calculation uses Assumption~\ref{ass:grt} and that the derivatives $\ik^{(m)}(w)$, $|m|\le M$, where $M>n/2$, are bounded. Therefore
\be\label{var der}
\begin{split}
&\Eb(\pa_{\chx}^m N_\e^{\text{rec}}(x_0+\e\chx))^2\\
&= \e_y^{n-N}\sum_{(y_j,z_k)\in\vs} \big[\e^\ga \pa_{\chx}^m F_\e(x,y_j,z_k)\big]^2\sigma^2(y_j,z_k)\\
&\le c \e^{n-N}\sum_{(y_j,z_k)\in\vs} \big[1+|k-(\Psi(x_0,y_j)/\e)|\big]^{-2(N+\ga)}\le c,\  \chx\in D.
\end{split}
\ee
This implies that $\Eb \Vert N_\e^{\text{rec}}(x_0+\e\chx))\Vert_{W^{M,2}(D)}^2\le c$ for all $\e>0$.
By the Chebyshev inequality,
\be\label{cheb}
\BP(N_\e^{\text{rec}}(x_0+\e\chx)\not\in\Gamma_\de)=\BP(\Vert N_\e^{\text{rec}}(x_0+\e\chx)\Vert_{W^{M,2}(D)}^2 \ge 1/\de)\le c\de.
\ee
Therefore $(N_\e^{\text{rec}}(x_0+\e\chx))$, $0<\e\ll1$, is a tight sequence. 

By Theorem~\ref{kh_thm}, $\nrec(x_0+\e\chx)\to N^{\text{rec}}(\chx)$, $\chx\in\bar D$, in distribution as $C$-valued random variables. Since $C$ is a complete metric space, it follows that $N^{\text{rec}}\in C$ has continuous sample paths with probability 1.

By the linearity of the expectation, $N^{\text{rec}}(\chx)$ is a zero mean GRF. To completely characterize this GRF, we calculate its covariance $\text{Cov}(\chx,\chy)=\lim_{\e\to0}\Eb(N_\e^{\text{rec}}(\chx)N_\e^{\text{rec}}(\chy))$,  $\chx,\chy\in D$. In fact, essentially this has already been done in the proof of Lemma~\ref{lem:Lyapunov_nd}: equation \eqref{recon dttpr 4} implies \eqref{Cov main}.

\section{Numerical experiments}\label{sec:numerics}

For a numerical experiment we consider cone beam local tomography reconstruction, see subsection~\ref{ssec:CBT} and Figure~\ref{fig:cbct}. Recall that the detector coordinates are $(u,v)\in\br^2$ and the source coordinate is $s\in[0,2\pi)$, see \eqref{source}, \eqref{detector}. The data are given at the points:
\be \label{data exper}
s=\Delta s j,\ u=\e k_1,\ v=\e k_2,\ j,k_1,k_2\in\BZ.
\ee

To reconstruct an image we backproject the second derivative of the cone beam transform along detector rows (i.e., along $u$). In the continuous case the reconstruction formula is as follows.
\be\label{recon CB}\bs
\frec(x)=&\int_0^{2\pi}\pa_u^2 g(s,U(x,s),V(x,s))\dd s\\
=&\int_0^{2\pi}\iint_{\br^2}\de^{\prime\prime}(U(x,s)-u)\de(V(x,s)-v) g(s,u,v)\dd u\dd v\dd s.
\end{split}
\ee
This is local tomography reconstruction \cite{luma, rk}. Representing the two delta-functions in terms of their Fourier transforms, we see that the complete symbol of the operator $\CA$ coincides with its principal symbol and $\ga=2$. 
Reconstruction from discrete data (consisting only of noise) is given by 
\be\label{recon CB discr}
\nrec(x)=\frac{\Delta s}{\e^2} \sum_{j,k}\ik^{\prime\prime}\bigg(\frac{U(x,s)-\e k_1}\e\bigg)\ik\bigg(\frac{V(x,s)-\e k_2}\e\bigg)\eta_{j,k}.
\ee
By \eqref{F def},
\be\label{CB Fe}\bs
F_\e(x,s,u^\prime,v^\prime)=&\int_{\br^2}\de^{\prime\prime}(U(x,s)-u)\de(V(x,s)-v)\\
&\hspace{1cm}\times \ik((u-u^\prime)/\e)\ik((v-v^\prime)/\e)\dd u\dd v\\
=&\frac1{\e^2}\ik^{\prime\prime}\bigg(\frac{U(x,s)-u^\prime}\e\bigg)\ik\bigg(\frac{V(x,s)-v^\prime}\e\bigg).
\end{split}
\ee
The function $\Phi$ is linear in $(u,v)$ (cf. \eqref{CB fns}) and the symbol of $\CA$ is independent of $(u,v)$, hence $\bar F_\e\equiv F_\e$. Then, by \eqref{bFG}, $G(\vartheta)=\ik^{\prime\prime}(\vartheta_1)\ik(\vartheta_2)$. Note that $G$ does not depend on $x$ and $s$. From \eqref{Cov main},
\be\label{GG conv}
(G\star G)(\vartheta)=\int_\br \ik^{\prime\prime}(\vartheta_1+r_1)\ik^{\prime\prime}(r_1)\dd r_1\int_\br \ik(\vartheta_2+r_2)\ik(r_2)\dd r_2.
\ee
Finally, the covariance function is given by (cf. \eqref{Cov main})
\be\label{covar LT}
C(\vartheta)=\int_0^{2\pi} (G\star G)\big(\pa_x (U(x_0,s),V(x_0,s))\vartheta\big)\sigma^2(s,U(x_0,s),V(x_0,s))\dd s,
\ee
where $\vartheta=\chx_1-\chx_2$. The derivatives $\pa_x U(x_0,s)$ and $\pa_x V(x_0,s)$ are computed explicitly using \eqref{det proj}.

Next we describe a numerical experiment to verify the obtained results, namely Theorem~\ref{GRF_thm}. The center of a local patch and two nearby points $x_k=x_0+\e\chx_k$, $k=1,2$ are selected as follows:
\be\label{points}\begin{split}
&x_0=(2.7,-3.1,0.8),\\ 
&\chx_1=(2.159,3.075, -0.418),\ \chx_2=(2.546, -2.974, 0.983).
\end{split}
\ee
The radius of the source trajectory is $R=10$ (see \eqref{source}). The values of $\eta_{j,k}$ are computed by the formula:
\be\bs\label{noise}
\eta_{j,k} =& (\e^2/\Delta s) h(s_j,u_{k_1},v_{k_2}) \nu_{j,k},\\ 
h(s,u,v)=&(1 + 0.5\sin(2s))(1 - 0.4\cos u)(1 + 0.6\sin v),
\end{split}
\ee
and $\nu_{j,k}$ are i.i.d. random variables drawn from the uniform distribution on $[-1,1]$. 
Comparing with \eqref{noise var}, we see that $(\e^2/\Delta s)^2=\e^{2\ga}/\Delta s^{n-N}$ (because $\e_y=\Delta s$, $\ga=2$ and $n-N=2$) and $\sigma^2(s,u,v)=(1/3)h^2(s,u,v)$. In the experiment, $2\cdot 10^4$ realizations of noisy sinograms, i.e. the sets $\{\eta_{j,k}\}$, have been used. We set 
\be
\Delta s=2\pi/500,\ \e=\Delta u=\Delta v=0.05.
\ee
The kernel function $\ik$ is obtained by convolving the interpolation kernel $(1-|s|)_+$ and the smoothing kernel $c_l(1-(s/a)^2)_+^l$:
\be
\ik(t)=\frac{(2l+1)!!}{2a(2l)!!}\int_{-1}^1 (1-|s|)\big[1-((t-s)/a)^2\big]_+^l\dd s,\ a=2.5,\ l=3.
\ee

Let us summarize the results of the experiment. At $x_0$, the sample variance is 0.488 and predicted variance is 0.485. The latter is obtained using \eqref{covar LT} with $\vartheta=0$. The observed probability density function (PDF), i.e., the normalized histogram, and predicted PDF are shown in Figure~\ref{fig:centerPDF1d}. The histogram is computed using 21 bins. The predicted PDF, which is a Gaussian with zero mean and variance 0.485, is computed by evaluating the Gaussian density at the center of each bin. The relative mismatch between the observed and predicted discretized PDFs (denoted $P_o$ and $P_p$, respectively) is found to be $\Vert P_o-P_p\Vert_1/\Vert P_p\Vert_1=0.021$. Here and below, $\Vert\cdot\Vert_1$ denotes the appropriate discrete $l^1$ norm (the sum of absolute values of the entries of a vector or matrix).

\begin{figure}[h]
{\centerline{
{\epsfig{file={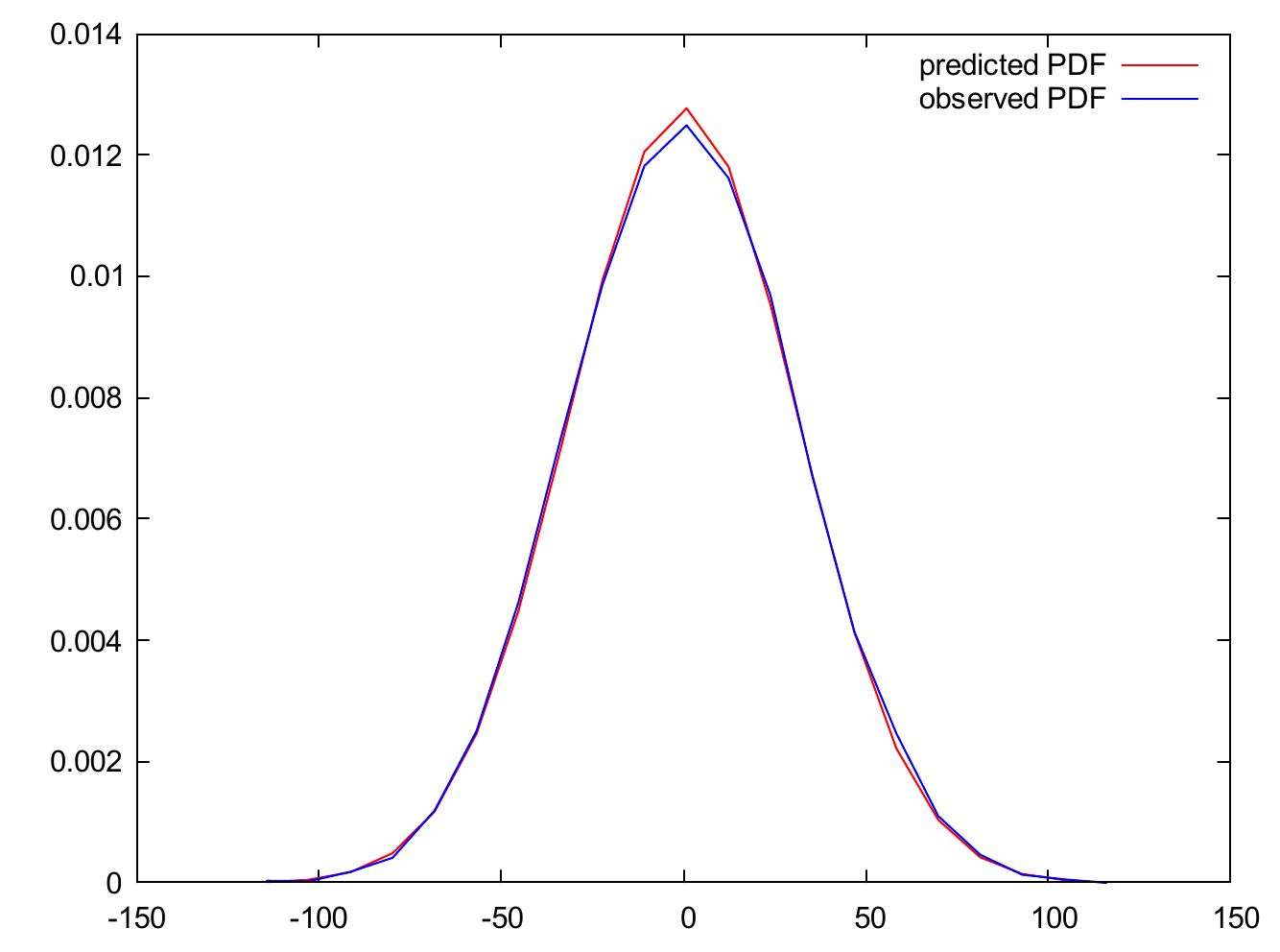}, width=10cm}}
}}
\caption{Observed PDF and predicted PDF for the random variable $\nrec(\check x=0)$ (i.e., at $x_0$ in the original coordinates). The $x$-axis represents the values of the random variable.}
\label{fig:centerPDF1d}
\end{figure}

Next we test the accuracy of our two-point estimates. For the two points $\chx_1$ and $\chx_2$ in \eqref{points}, the observed and predicted covariances and the relative error between them are found to be
\be\label{covars}\begin{split}
&\text{Cov}_o=\begin{pmatrix} 0.479 & 0.013\\
 0.013& 0.458\end{pmatrix},\ 
\text{Cov}_p=\begin{pmatrix}  0.485 & 0.011\\
 0.011 & 0.485\end{pmatrix},\\ 
&\Vert \text{Cov}_o-\text{Cov}_p\Vert_1/\Vert \text{Cov}_p\Vert_1=0.035.
 \end{split}
\ee
The diagonal entries of the covariance matrix $\text{Cov}_p$ are equal $C(0)$, and off-diagonal elements are equal $C(\chx_1-\chx_2)$. The observed 2D PDF and predicted 2D PDF are shown in Figure~\ref{fig:centerPDF2d}. The histogram is computed using $21\times21$ bins. The predicted PDF is computed by evaluating the predicted Gaussian density at the center of each bin. The relative mismatch between the observed and predicted discretized PDFs is found to be $\Vert P_o-P_p\Vert_1/\Vert P_p\Vert_1=0.079$, which is quite small as well.
\begin{figure}[h]
{\centerline{\hbox{
{\epsfig{file=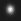}, width=4.5cm}}
{\epsfig{file={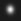}, width=4.5cm}}
{\epsfig{file={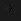}, width=4.5cm}}
}}}
\caption{Observed PDF (left), predicted PDF (center), and their difference (right) for the random vector $(\nrec(\check x_1),\nrec(\check x_2))$. The display range is [-2E-5,1.7E-4] for all three figures. The horizontal and vertical axes represent the values of $\nrec(\check x_1)$ and $\nrec(\check x_2)$, respectively.}
\label{fig:centerPDF2d}
\end{figure}

\appendix

\section{Proofs of two auxiliary results}\label{sec:}\label{sec:two aux res}

\subsection{Proof of \eqref{F-F1}}\label{ssec:ff1}
By \eqref{F v3} and \eqref{F1 def}, we have to estimate the quantity
\be\label{del A1 ker}
\begin{split}
&J:=(2\pi)^{-N}\int_{\br^N}\int_{W_\e(u,z^\prime)} \Delta a_1(u,\e\check w,\hat\xi)\ik\bigg(\frac{z(\e\check w,u)-z^\prime }\e\bigg)e^{i\hat\xi\cdot \check w}\dd \check w \dd\hat\xi,\\
&\Delta a_1(u,w,\hat\xi):=a_1(u,z(w,u),\hat\xi)-a_1(u,\Psi(u),\hat\xi),\ (u,z^\prime)\in\Lambda.
\end{split}
\ee
Suppose $|\check v|\le c$. By \eqref{triangle} and using Assumption~\ref{ass:grt}(2) and Assumption~\ref{ass:ker}, the analog of \eqref{G1 bnd} becomes
\be\label{del G1 bnd}\begin{split}
|H_\e(u,z^\prime ,\hat\xi)|\le & c\e (1+|\hat\xi|)^{\ga-M_1},\ \hat\xi\in\br^N.
\end{split}
\ee
Similarly to \eqref{F v4}, integration by parts to prove \eqref{del G1 bnd} involves $\pa_{\check w}$ rather than $\pa_w$. Integrating with respect to $\hat\xi$ gives $|J|\le c\e$.

If $|\check v|\ge c$, \eqref{part 2} implies
\be\label{J1st}
J=\int_{W_\e(u,z^\prime)} \big[A_1(u,z^\prime ,\check w)-A_1(u,\Psi(u),\check w)\big]\ik\bigg(\frac{z(\e \check w,u)-z^\prime }\e\bigg)\dd \check w.
\ee
By Assumption~\ref{ass:grt}(2), the kernel $A_1$ depends smoothly on $z^\prime$. Therefore \eqref{A1 homog} yields:
\be\label{del F2 bnd_1}\begin{split}
|J|\le c \e|\check v|(1+|\check v|)^{-(N+\ga)}\text{  if  }  |\check v|\ge c.
\end{split}
\ee
The desired assertion \eqref{F-F1} now follows.

\subsection{Proof of \eqref{F1-Fbar}}\label{ssec:f1fb}

We need to estimate the quantity
\be\label{del Abar ker}
\begin{split}
J:=&(2\pi)^{-N}\int_{\br^N} \Delta H_\e(u,z^\prime,\hat\xi)\dd\hat\xi,\ (u,z^\prime)\in\Lambda,\\
\Delta H_\e(u,z^\prime,\hat\xi):=&\int_{\br^N} a_1(u,\Psi(u),\hat\xi)\\
&\times \bigg[\ik\bigg(\frac{z(\e\check w,u)-z^\prime }\e\bigg)-\ik\bigg(\frac{\bar z(\e\check w,u)-z^\prime }\e\bigg)\bigg]e^{i\hat\xi\cdot \check w}\dd \check w.
\end{split}
\ee
Using \eqref{del phi} and Assumptions~\ref{ass:grt}(2) and \ref{ass:ker}, the analogs of \eqref{F v4} and \eqref{G1 bnd} become
\be\label{del F v4}\begin{split}
|\Delta H_\e(u,z^\prime ,\hat\xi)|\le &
c(1+|\hat\xi|)^{\ga-M_1}  \int_{|\check w|\le c} \big[\min(1,\e|\check w|^2)+\e|\check w|\big]\dd \check w\\
\le & c\e (1+|\hat\xi|)^{\ga-M_1},\ \hat\xi\in\br^N,|\check v|\le c.
\end{split}
\ee
Integrating with respect to $\hat\xi$ gives $|J|\le c\e$, $|\check v|\le c$. As before, we need $\ik\in C_0^{M_1+1}(\br^N)$ for this to work. Suppose now $|\check v|\ge c$. Similarly to \eqref{J1st},
\be\label{J2nd}
J=\int_{\br^n} A_1(u,\Psi(u),\check w)\bigg[\ik\bigg(\frac{z(\e \check w,u)-z^\prime }\e\bigg)-\ik\bigg(\frac{\bar z(\e \check w,u)-z^\prime }\e\bigg)\bigg]\dd \check w.
\ee
By \eqref{del phi}, the analog of \eqref{F2 bnd} becomes
\be\label{del F2 bnd_2}\begin{split}
|J|\le c \big[\min(1,\e|\check v|^2)+\e|\check v|\big](1+|\check v|)^{-(N+\ga)}\text{  if  }  |\check v|\ge c.
\end{split}
\ee

Even though the integrals \eqref{del Abar ker} and \eqref{J2nd} involve $\ik$ of two different arguments, the domain of integration of each of them has the same key properties as $W_\e(u,z^\prime)$: it is a uniformly bounded set, it is confined to a ball $|\check w|\le c$ if $|\check v|\le c$, and it is a subset of $\{\check w\in\br^N:|\check w|\ge c|\check v|\}$ if $|\check v|\ge c$. 

The desired assertion \eqref{F1-Fbar} now follows.

\section{Proof of Lemma~\ref{lem:generic cond}}\label{sec:generic cond prf}
Since $x_0$ is fixed, in this and the remaining appendices we omit $x_0$ from most notations. 

Let $d_0(\xi) $ be the upper box-counting dimension of $Y_1(\xi)$, see Definition~\ref{def:box-dim} and \eqref{Ydef}. Then $Y_1(\xi)$ can be covered by no more than $(1/\de_1)^d$ (hyper)cubes $\mathcal B_k\subset\br^{n-N}$ with side length $\de_1$ for any $d>d_0(\xi)$ and $\de_1>0$ sufficiently small. Replace each $\mathcal B_k$ with a (slightly) larger open cube $\mathcal B_k^{(1)}$ (say, twice the volume), which has the same center and orientation as $\mathcal B_k$. The sum of their $(n-N)$-dimensional volumes satisfies 
$$
\sum_k \text{Vol}(\mathcal B_k^{(1)})\le c\de_1^{(n-N)-d}\to0,\ \de_1\to0,
$$
because we can take $d_0(\xi)<d<n-N$. Find $\de_1>0$ so that $\sum_k \text{Vol}(\mathcal B_k^\prime)\le \de/2$. By construction, the interior of $\cup_k \mathcal B_k^{(1)}$ is a neighborhood of $Y_1(\xi)$. 

Pick any $u\in\BZ^{n-N}$ and $\xi\in\br^N\setminus0$. By construction, $\text{det}(\pa_y^2(\xi\cdot\Psi))$ is bounded away from zero on the set $\CY\setminus (\cup_k\mathcal B_k^{(1)})$. Therefore, there are finitely many solutions to the equation $\pa_y(\xi\cdot\Psi(y))=u$ in $\CY\setminus (\cup_k\mathcal B_k^{(1)})$. Cover these solutions by finitely many sufficiently small open cubes $\mathcal B_k^{(2)}$, which satisfy $\sum_{k}\text{Vol}(\mathcal B_k^{(2)})\le \de/2$. Let $\mathcal B$ denote the union of all the cubes $\mathcal B_k^{(1)}$, $\mathcal B_k^{(2)}$.
By construction, the interior of $\mathcal B$ is a neighborhood of the set of solutions to $\pa_y(\xi\cdot\Psi(y))=u$, $y\in\CY$. Therefore,
\be
\sum_{k}\text{Vol}(\mathcal B_k^{(1)})+\sum_{k}\text{Vol}(\mathcal B_k^{(2)})\le \de,\ \inf_{y\in\CY\setminus \mathcal B} |\pa_y(\xi\cdot\Psi(y))-u|=:p>0,
\ee
and the lemma is proven.

\section{Proof of Lemma~\ref{lem:Lyapunov}}\label{sec:variance}
Let $D_\e$ be the expression in brackets in the denominator in \eqref{main lim L1d}. By \eqref{noise var},
\be\bs
&\kappa_\e^2\sum_{z_k\in\CZ}\big|(G(y_j,k-b_j)+R_{j,k})^2-G^2(y_j,k-b_j)\big|\Eb \eta_{j,k}^2\\
&\le c\e^{n-N}\sum_{z_k\in\CZ}\big[R_{j,k}^2+|G(y_j,k-b_j)||R_{j,k}|\big].
\end{split}
\ee
See \eqref{recon} for the definitions of $\kappa_\e$ and $b_j$. Using \eqref{G bnd} and \eqref{R bnd},  straightforward calculations shows that 
\be\label{sum R bnd}\bs
\sum_{z_k\in\CZ}R_{j,k}^2\le c\e^\de,\ 
\sum_{z_k\in\CZ}|G(y_j,k-b_j)||R_{j,k}|\le c\e^\de
\end{split}
\ee
for some $\de>0$ as long as $\ga>-N/2$. Therefore, by \eqref{noise var}
\begin{equation}\label{Deps v1}
D_\e=\e_y^{n-N}\sum_{(y_j,z_k)\in\vs} G^2(y_j,k-b_j)\sigma^2(y_j,z_k)+O(\e^\de).
\end{equation}


Define 
\begin{equation}\label{psi_a_b}\begin{split}
\psi(y,r):=&\sum_{k\in\BZ^N}  G^2(y,k-r),\ y\in\CY,r\in\br^N.
\end{split}
\end{equation} 
By \eqref{G bnd} the series is absolutely convergent. It is easy to see that
\begin{align}\label{Psi per}
\psi(y,r+m)=\psi(y,r),\  y\in\CY,r\in\br^N,m\in\BZ^N.
\end{align}
Replace $z_k$ with $\Psi(y_j)$ in the arguments of $\sigma^2$ in \eqref{Deps v1} to obtain
\begin{equation}\label{Deps}
D_\e=\e_y^{n-N}\sum_{y_j\in\CY} \psi(y_j,b_j)\sigma^2(y_j,\Psi(y_j))+O(\e^\de).
\end{equation}
Indeed, consider the sum with respect to $k$. By the definition of $b_j$ in \eqref{recon} we have
$z_k-\Psi(y_j)=\e(k-b_j+O(1))$. By \eqref{G bnd} and Lipschitz continuity of $\sigma$:
\be\bs\label{sigma mod}
&\sum_{|k|\le O(1/\e)} G^2(y,k-b_j)|\sigma^2(y,\e k)-\sigma^2(y,\Psi(y_j))|\\
&\le c\sum_{|k|\le O(1/\e)} \e(1+|k-b_j|)(1+|k-b_j|)^{-2(N+\ga)}=O(\e^a)\to0,
\end{split}
\ee
where $a=\min(1,2\ga+N)$. The assumption $\ga>-N/2$ implies $a>0$. Summing over $y_j\in\CY$ and accounting for the factor $\e_y^{n-N}$ proves the claim.

Using arguments similar to \cite{Katsevich_2024_BV}, we prove in Section~\ref{sec: 2nd DTB} the following result
\begin{align}\label{lim De}
\lim_{\e \to 0}D_\e&=\int_{\CY}\bigg(\int_{[0,1]^N} \psi(y,r) \dd r\bigg)\sigma^2(y,\Psi(y)) \dd y.
\end{align} 
Furthermore, 
\begin{equation}\label{lim psi int}\bs
\int_{[0,1]^N} \psi(y,r)\dd r&=\sum_{k\in\BZ^N} \int_{[0,1]^N}G^2(y,k-r)\dd r\\
&=\int_{\br^N}G^2(y,r)\dd r=:H(y).
\end{split}
\end{equation}
By Parseval's theorem and \eqref{bFG},
\be\bs
\int_{\br^N}G^2(y,r)\dd r = & (2\pi)^{-N}\int_{\br^N}\big| \tilde G(y,\hat\mu)\big|^2\dd\hat\mu\\
=&(2\pi)^{-N}\int_{\br^N}| a_2(y,\hat\mu)\tilde\ik(\hat\mu)|^2\dd\hat\mu,
\end{split}
\ee
where $\tilde G(y,\hat\mu)$ denotes the $N$-dimensional Fourier transform of $G(y,r)$ with respect to $r$. Returning the dependence on $x_0$, we get using Assumption~\ref{ass:x0}(2)
\begin{align}\label{A.6}
\lim_{\e \to 0}D_\e=\int_{\CY} H(x_0,y)\sigma^2(y,\Psi(x_0,y)) \dd y>0.
\end{align}


Consider now the numerator in \eqref{main lim L1d}. Using \eqref{noise var}, \eqref{G bnd}, and \eqref{R bnd}, we obtain after simple calculations:
\be\label{sum G R}
\kappa_\e^3\sum_{z_k\in\CZ}\lvert G(y_j,k-b_j)+R_{j,k}\rvert^3 \Eb\lvert \eta_{j,k}\rvert^3
=o(\e^{n-N}),
\ee
where the small-$o$ is uniform with respect to $j$. Then the numerator in \eqref{main lim L1d} is bounded by 
\begin{align}\label{A.7}
\sum_{y_j\in\CY}o(\e^{n-N})=o(1).
\end{align}
Together with \eqref{A.6} this proves Lemma~\ref{lem:Lyapunov}.

\section{Proof of \eqref{lim De}}\label{sec: 2nd DTB}

Let $D_\e^\prime$ denote the first term on the right in \eqref{Deps}. We can write $D_\e^\prime$ in the following form
\begin{equation}\label{Deps 1}\bs
&D_\e^\prime=\e_y^{n-N}\sum_{y_j\in\CY} g\big(y_j,\pa_x\Psi(y_j)\cdot\chx+(\Psi(y_j)/\e)\big),\\ 
&g(y,r):=\psi(y,r)\sigma^2(y,\Psi(y)),
\end{split}
\end{equation}
cf. the definition of $b_j$ in \eqref{recon}. By \eqref{Psi per}, $g(y,r)=g(y,r+m)$ for any $y\in\CY,r\in\br^N$, and $m\in\BZ^N$. 

The proof of \eqref{lim De} is done in two steps. In section~\ref{ssec:ud} we prove several auxiliary results related to the uniform distribution property of a collection of points. The final calculation of the limit then follows easily and this is done in section~\ref{ssec:limDe}. 

\subsection{Step I: using the uniform distribution property}\label{ssec:ud}

Given two vectors $a,b\in\br^N$ such that $a_l<b_l$, $1\le l\le N$, we define the hyperrectangle $[a,b]:=[a_1,b_1]\times\dots\times[a_N,b_N]$. Recall that $e(t)=\exp(2\pi i t)$, $t\in\br$.


\begin{theorem}\label{thm:weil} Let $U\subset\br^{n-N}$ be a bounded domain and $g\in C^2(\br^N)$ be a function, which satisfies $g(r)=g(r+m)$ for any $r\in\br^N$ and $m\in\BZ^N$. Suppose the points $v_j(\e)\in\br^N$, $j\in\BZ^{n-N}$, $\e j\in U$, $0<\e\ll1$, have the property
\be
\lim_{\e\to 0}\e^N \sum_{\e j\in U}e(m\cdot v_j(\e))=0 
\ee
for any $m\in \mathbb Z^N\setminus0$. Then 
\be
\lim_{\e\to 0}\e^N \sum_{\e j\in U}g(v_j(\e))=|U|\int_{[0,1]^N}g(r)\dd r.
\ee
\end{theorem}

Theorem~\ref{thm:weil} is well-known in the literature (see \cite[Theorems 6.1, 6.2]{KN_06}). Usually such a result is proven for a sequence of points $v_j$, $j\ge 1$. We modify the statement to allow for the points in the sequence,  $v_j(\e)$, to depend on $\e>0$. Nevertheless, the proof in this slightly more general case is the same and based on expanding $g$ in the Fourier series.

Let $\langle r \rangle$ denote the distance from $r\in\br$ to the nearest integer. We also denote $f_l^\prime:=\pa_{y_l}f(y)$, the partial derivative of $f$ with respect to the $l$-th coordinate of $y$.

\begin{lemma}\label{lem:udpts-e} Given a hyperrectangle $[a,b]\subset\br^N$, consider a real-valued $f\in C^2([a,b])$ such that 
\be\label{incl-e}
\langle f_l^\prime(y)\rangle \ge \de,\ y\in [a,b],
\ee
for some $\de>0$ and for some $l=l_0$, $1\le l_0\le N$. Then 
\be\label{wlcr-e}
I_\e:=\e^N\sum_{\e j\in[a,b]}e(f(\e j)/\e)=O(\e^{1/3}),\ \e\to0.
\ee
\end{lemma}

\begin{proof} 
Partition $[a,b]$ into hypercubes with side length $\e^{2/3}$:
\be\label{intervs-e}\bs
&\Delta_k:=[y_k,y_{k+\vec 1}],\ y_k:=a+k\e^{2/3},\\ 
&k=(k_1,\dots,k_N),\ 0\le k_l<(b_l-a_l)\e^{-2/3},
\end{split}
\ee
where $\vec 1:=(1,\dots,1)\in\br^N$. 
To clarify, $y_k\in\br^N$ are points, and the $l$-th coordinate of $y_k$ is denoted $y_{k,l}$.
First, we show that
\be\label{frst-est-e}\begin{split}
I_\e(k):&=\e^{N/3}\bigg|\sum_{\e j\in \Delta_k}e\left(\frac{f(\e j)}\e\right)
-\sum_{\e j\in \Delta_k}e\left(\frac{f(y_k)+ f^\prime(y_k)(\e j-y_k)}\e\right)\bigg|
\to 0,\\ 
\e&\to0,\ \Delta_k\subset [a,b].
\end{split}
\ee
Since
\be\label{taylor2-e}
f(\e j)=f(y_k)+ f^\prime(y_k)(\e j-y_k)+O(\e^{4/3}),\ \e j\in\Delta_k\subset [a,b],
\ee
it follows from \eqref{frst-est-e}
\be\label{Ismall-e}
I_\e(k)\le \e^{N/3}\sum_{\e j\in \Delta_k}\left|e\big(O\big(\e^{1/3}\big)\big)-1 \right|=
O\big(\e^{1/3}\big),\ \Delta_k\subset [a,b].
\ee
The big-$O$ terms in \eqref{taylor2-e}, \eqref{Ismall-e} are uniform in $k$ since $\max_{y\in[a,b]}\Vert \pa_y ^2 f(y)\Vert<\infty$. Here $\Vert\cdot\Vert$ denotes any matrix norm.

To prove \eqref{wlcr-e}, partition the sum and use \eqref{frst-est-e}, \eqref{Ismall-e} to obtain the estimate:
\be\label{Jest1-e}\begin{split}
|I_\e|&\le \e^N \sum_k\bigg|\sum_{\e j\in \Delta_k}e\left(\frac{f(y_k)+f^\prime(y_k)(\e j-y_k)}\e\right)\bigg|+O\big(\e^{1/3}\big)\\
&= \e^{2N/3} \sum_k\prod_{l=1}^N \bigg|\e^{1/3}\sum_{\e j_l\in [y_{k,l},y_{k+1,l}]} e\big(f_l^\prime(y_k)j_l\big)\bigg|+O\big(\e^{1/3}\big).
\end{split}
\ee
Here and below
\be\label{sum k}
\sum_k:=\sum_{k_1=0}^{(b_1-a_1)\e^{-2/3}-1}\dots \sum_{k_N=0}^{(b_N-a_N)\e^{-2/3}-1}.
\ee

Next, we use the inequality
\be\label{keyineq-e}
\bigg| \sum_{n=1}^J e(r n)\bigg|\le \min\big(J,(2\langle r \rangle)^{-1}\big).
\ee
 By \eqref{incl-e}, 
$\langle f_l^\prime(y_k)\rangle\ge \de$, $l=l_0$.
Using \eqref{keyineq-e} in \eqref{Jest1-e} with
\be\label{consts-e}
r=f_l^\prime(y_k),\ J=O(\e^{-1/3}),\ l=l_0,
\ee
gives
\be\label{midest-e}
\e^{1/3}\bigg|\sum_{\e j_l\in [y_{k,l},y_{k+1,l}]} e\big(f_l^\prime(y_k)j_l\big)\bigg|\le c \min\big(1,\e^{1/3}/(2\de)\big),\ l=l_0.
\ee
Recall that $y_{k,l}$ is the $l$-th coordinate of $y_k$.
Multiply both sides of \eqref{midest-e} with all the remaining factors ($l\not=l_0$) in \eqref{Jest1-e}, use that all these factors are bounded, sum the resulting inequalities over $k$ as in \eqref{sum k}, and multiply by $\e^{2N/3}$. This gives the same upper bound as in \eqref{midest-e}. Hence $I_\e=O(\e^{1/3})$ and the lemma is proven.
\end{proof}

\begin{corollary}\label{cor:udpts-v2-e} Given a bounded domain $U\subset\br^N$, pick a real-valued $f\in C^2(U)$. Suppose for every $u\in\BZ^N$ and every $\de>0$ sufficiently small, 
there exist $p>0$ and a finite collection of cubes $\mathcal B_l\subset\br^n$ such that
\be\label{cor_cond1_2}\begin{split}
&\{y\in U:\, |f^\prime(y)-u|\le p\}\subset \cup_l\mathcal B_l,\ 
\sum_l\text{Vol}(\mathcal B_l)\le\de.
\end{split}
\ee
Then 
\be\label{wlcr-e-v2}
\e^N\sum_{\e j\in U}e(f(\e j)/\e)\to 0,\ \e\to0.
\ee
\end{corollary}

\begin{proof} 
Pick $\de>0$ arbitrarily small. Since the set $f^\prime(y)$, $y\in U$, is bounded, there are finitely many $u\in\mathbb Z^N\cap f^\prime(U)$. Therefore, by the assumptions of the corollary, there exist $p>0$ and a finite collection of cubes $\mathcal B_l$ such that
\be\label{cond1 conseq}\begin{split}
&\big\{y\in U:\, |f^\prime(y)-u \big|\le p \text{ for some }u\in \BZ^N\big\}\subset \cup_l\mathcal B_l,\quad
\sum_l\text{Vol}(\mathcal B_l)\le\de.
\end{split}
\ee
Denote $U(\de):=U\setminus \cup_l\mathcal B_l$. By construction, 
\be\label{Vde1}
U(\de)\subset\{y\in U:\, |f^\prime(y)-u|>p,\ \forall u\in\mathbb Z^N\}. 
\ee
Approximate $U(\de)$ by a union of finitely many pairwise nonintersecting cubes $\Delta_l$: 
$$
\cup_l \Delta_l \subset U(\de),\ \text{Vol}(U(\de)\setminus \cup_l \Delta_l)<\de. 
$$
By construction, $f(y)$ satisfies the assumptions of Lemma~\ref{lem:udpts-e} on each of the $\Delta_l$. This follows from \eqref{Vde1} and because $|f^\prime(y)-u|>p$ implies that $|\pa_{y_\iota}f(y)-u_l|>p/N^{1/2}$ for at least one $\iota=1,2,\dots,N$. By Lemma~\ref{lem:udpts-e}, 
\be
\e^N\sum_{\e j\in \Delta_l} e(f(\e j)/\e)\to 0,\ \e\to0,
\ee
for each $\Delta_l$, and there are finitely many of them. By choosing $\de>0$ small, $\cup_l \Delta_l$ can be as close to $U$ as we like, and the desired assertion follows.
\end{proof}

In what follows we apply Corollary~\ref{cor:udpts-v2-e} with $N$ replaced by $n-N$ (i.e., $U\subset\br^{n-N}$) and $\e$ replaced by $\e_y$.

\begin{corollary}\label{cor:udpts-2d-e} Fix a bounded domain $U\subset\br^{n-N}$ and a $C^2$ function $\Psi:U\to \br^N$. Let $g\in C^2(\br^N)$ be a function, which satisfies $g(r)=g(r+m)$ for any $r\in\br^N$ and $m\in\BZ^N$. Suppose the function $m\cdot \Psi(y):U\to \br$ satisfies the assumptions of Corollary~\ref{cor:udpts-v2-e} for any $m\in\mathbb Z^N\setminus0$. Then
\be
\lim_{\e_y\to 0}\e_y^{n-N} \sum_{\e_y j\in U}g(\Psi(\e_y j)/\e_y)=\text{Vol}(U)\int_{[0,1]^N}g(r)\dd r.
\ee
\end{corollary}

\begin{proof} By setting $f(y):=m\cdot \Psi(y):U\to \br$, Corollary~\ref{cor:udpts-v2-e} implies  
that
\be
\lim_{\e_y\to 0}\e_y^{n-N} \sum_{\e_y j\in U}e(m\cdot \Psi(\e_y j)/\e_y)=0,\ \forall m\in\BZ^N\setminus0.
\ee
The desired result follows from Theorem~\ref{thm:weil} by setting $v_j(\e_y)=\Psi(\e_y j)/\e_y$.
\end{proof}

\subsection{Step II: computing $\lim_{\e\to0}D_\e$}\label{ssec:limDe}
Fix some $0<\de\ll 1$ and partition $\CY$ into $L(\de)$ domains $U_l$, $l=1,\dots,L(\de)$, such that the diameter of each domain does not exceed $\de$. Pick any $\tilde y_l\in U_l$ for each $l$. Clearly,
\begin{equation}\label{Deps 2}
D_\e=\sum_{l=1}^{L(\de)}\text{Vol}(U_l)\bigg[\frac{\e_y^{n-N}}{\text{Vol}(U_l)}\sum_{y_j\in U_l} g(\tilde y_l,\Psi(y_j)/\e_y)\bigg]+O(\de).
\end{equation}
Applying Corollary~\ref{cor:udpts-2d-e} to each expression in brackets (i.e., with $U$ replaced by $U_l$) we obtain
\begin{equation}\label{Deps 2 lim}
\lim_{\e\to0}D_\e=\sum_{l=1}^{L(\de)}\text{Vol}(U_l)\int_{[0,1]^{N}} g(\tilde y_l,r)\dd r+O(\de).
\end{equation}
By Lemma~\ref{lem:generic cond}, Corollary~\ref{cor:udpts-2d-e} does apply. Indeed, for Corollary~\ref{cor:udpts-2d-e} to apply, the main property is that for any $m\in\BZ^N$ and $0<\de\ll1$ the set
\be\label{req prop Psi}\begin{split}
\{y\in U_l:\, |\pa_y(m\cdot \Psi(y))-u|\le p\}
\end{split}
\ee
is contained in a union of finitely many cubes with the total volume not exceeding $\de$. Lemma~\ref{lem:generic cond} proves a slightly stronger property for all of $\CY$, where instead of a vector $m\in\BZ^N\setminus0$ we have $\xi\in\br^N\setminus0$. 

Since $\de>0$ can be arbitrarily small, the proof is finished.

\section{Outline of proof of Lemma~\ref{lem:Lyapunov_nd}}\label{sec:prf outline}
The proof of Lemma~\ref{lem:Lyapunov_nd} is similar to that of Lemma~\ref{lem:Lyapunov}, so here we only highlight key new points. 

The denominator in \eqref{main lim Lnd} is $D_\e^{3/2}$, where (cf. \eqref{recon dttpr 1} for the definitions of $b_j^{(l)}$ and $R_{j,k}^{(l)}$)
\begin{equation}\label{recon dttpr 2}
\begin{split}
D_\e=&\e_y^{n-N}\sum_{(y_j,z_k)\in\vs} \biggl[\sum_{l=1}^L \theta_l \big(G\big(y_j,k-b_j^{(l)}\big)+R_{j,k}^{(l)}\big)\biggr]^2\sigma^2(y_j,\Psi(y_j))\\
&\hspace{1cm}+O(\e^a),\ a=\min(1,2\ga+N).
\end{split}
\end{equation} 
Here we have used the usual substitution $z_k\approxeq \Psi(y_j)$ (cf. \eqref{Deps}, \eqref{sigma mod}) and \eqref{noise var}.
Estimating the contribution of the remainder $R_{j,k}^{(l)}$ similarly to \eqref{sum R bnd} and passing to the limit as in \eqref{lim De} gives
\begin{equation}\label{recon dttpr 3}
\begin{split}
\lim_{\e\to0}D_\e=&\int_{\CY}\int_{\br^{N}} f^2(r,y)\dd r\, \sigma^2(y,\Psi(y))\dd y,\\
f(r,y):=&\sum_{l=1}^L  \theta_l G\big(y,r-\pa_x\Psi(y)\cdot\chx_l\big).
\end{split}
\end{equation} 
Multiplying out the sum in the definition of $f$, we can write 
\begin{equation}\label{recon dttpr 4}
\begin{split}
\lim_{\e\to0}D_\e=&\sum_{l_1=1}^L\sum_{l_2=1}^L  \theta_{l_1}\theta_{l_2}C(\chx_{l_1}-\chx_{l_2}),\\
C(\vartheta):=&\int_{\CY}(G\star G)\big(y,\pa_x\Psi(y)\cdot \vartheta\big)\sigma^2(y,\Psi(y))\dd y,\\
(G\star G)(y,\vartheta):=&\int_{\br^N}G(y,\vartheta+r)G(y,r)\dd r.
\end{split}
\end{equation} 
Observe that if $L=1$, $\theta_1=1$, and $\vartheta=0$, then \eqref{recon dttpr 4} coincides with \eqref{lim De}, \eqref{lim psi int}. By \eqref{bFG},
\begin{equation}\label{f alt}
\begin{split}
f(r,y)=(2\pi)^{-N}\int_{\br^N} a_2(y,\hat\mu)\tilde\ik(\hat\mu) \bigg[\sum_{l=1}^L  \theta_l e^{-i\hat\mu\cdot \pa_x\Psi(y)\chx_l}\bigg]e^{i\hat\mu\cdot r}\dd \hat\mu.
\end{split}
\end{equation} 

Suppose the limit in \eqref{recon dttpr 3} equals zero. By Assumption~\ref{ass:x0}(2), this implies that $f(r,y)\equiv0$ for all $r\in\br^N$ and $y\in V$. The set $V$ is introduced in Assumption~\ref{ass:x0}(2). Also, $\ik$ is compactly supported, so $\tilde\ik$ is analytic and cannot be zero on an open set. Therefore
\begin{equation}\label{zero expon}
\begin{split}
\sum_{l=1}^L  \theta_l e^{-i\hat\mu\cdot \pa_x\Psi(y)\chx_l}\equiv 0,\ \hat\mu\in\br^N,\ y\in V.
\end{split}
\end{equation} 

Consider the vectors $e_p:=\chx_{l_1}-\chx_{l_2}\in\br^n$, $1\le l_1,l_2\le L$, $l_1\not=l_2$, and the sets $V_p:=\{y\in V: \pa_x\Psi(y) e_p=0\}$. Recall that the first argument of $\Psi$ is $x=x_0$, which we omitted for simplicity. The $\chx_l$ are distinct, so $e_p\not=0$ for each $p$. By \eqref{Bolker},  each set $V_p$ has measure zero. There are finitely many $V_p$, so their union is not all of $V$. Therefore there exists $y_0\in V\setminus(\cup V_p)$ such that $\pa_x\Psi(y_0) e_p\not=0$ for each $p$. Since the union of any finite number of the proper subspaces $(\pa_x\Psi(y_0) e_p)^\perp\subset\br^N$ cannot be all of $\br^N$, there exists $\hat\mu\in\br^N$ such that $\hat\mu\cdot\pa_x\Psi(y_0) e_p\not=0$ for any $p$.

In summary, there exist $y_0\in V$ and $\hat\mu$ such that $s_l:=\hat\mu\cdot \pa_x\Psi(y_0)\chx_l$, $1\le l\le L$, are pairwise distinct. Replace $\hat\mu$ with $\la\hat\mu$ in \eqref{zero expon}, differentiate with respect to $\la$ multiple times, and set $\la=0$. This gives
\begin{equation}\label{equations}
\sum_{l=1}^L  \theta_l s_l^p=0,\ p=0,1,\dots,L-1.
\end{equation} 
Since all $s_l$ are pairwise distinct, using the properties of the Vandermond determinant we conclude from \eqref{zero expon} and \eqref{equations} that all $\theta_l=0$. This contradiction proves that the limit in \eqref{recon dttpr 3} is not zero unless $\vec\theta=0$. 


An argument to show that the limit of the numerator in \eqref{main lim Lnd} is bounded is very similar to \eqref{sum G R}, \eqref{A.7}.

\bibliographystyle{abbrv}
\bibliography{My_Collection, refs}
\end{document}